\documentclass[9pt,		
paper=a4,
bibliography=totoc,			
oneside,
openany,
dvipsnames,
]{scrartcl}

\usepackage[english]{babel}
\usepackage{algorithmicx,amsthm,enumitem,color,hyperref,graphicx,caption,subcaption,float,blindtext,import,todonotes,amsmath,tabu,amsfonts,amssymb,mathtools,nicefrac,colonequals}
\usepackage[autostyle]{csquotes}

\usepackage[ruled,vlined,linesnumbered]{algorithm2e}

\usepackage{tikz}
\usepackage{tikz-3dplot}
\usetikzlibrary{shapes.misc,decorations.pathreplacing,calc}
\usepackage[
bottom=3.5cm
]{geometry}
\newtheoremstyle{bla}{}{}{\normalfont}{}{\bfseries}{}{0.5em}{}
\newtheoremstyle{bla2}{}{}{\itshape}{}{\bfseries}{}{0.5em}{}
\theoremstyle{bla2}
\newtheorem{theorem}{Theorem}[section]
\newtheorem{proposition}[theorem]{Proposition}
\newtheorem{lemma}[theorem]{Lemma}
\newtheorem{corollary}[theorem]{Corollary}

\theoremstyle{bla}
\newtheorem{definition}[theorem]{Definition}

\newtheorem{observation}[theorem]{Observation}
\newtheorem{example}[theorem]{Example}
\newtheorem{assumption}[theorem]{Assumption}

\usepackage{graphicx}
\usepackage{nicefrac}
\usepackage{mathtools}

\usepackage{stmaryrd}

\usepackage[normalem]{ulem}
\newcommand{\R}{\mathbb{R}}

\newcommand{\Z}{\mathbb{Z}}
\newcommand{\N}{\mathbb{N}}
\let\int\relax
\DeclareMathOperator{\int}{int}
\DeclareMathOperator{\conv}{conv}

\newcommand{\A}{\mathcal{A}}

\newcommand{\ALG}{\mathtt{ALG}}

\newcommand{\lmin}{\lambda^{\min}}

\newcommand{\Lapprox}{\Lambda^{\textnormal{compact}}}
\newcommand{\bLapprox}{\bar{\Lambda}^{\textnormal{compact}}}
\newcommand{\Wapprox}{W^{\textnormal{cone}}}

\newcommand{\Wapproxone}{W^{\textnormal{compact}}}
\newcommand{\bWapproxone}{\bar{W}^{\textnormal{compact}}}
\newcommand{\Grid}{\Lambda^{\textnormal{Grid}}}

\newcommand{\proj}{\textnormal{proj}}
\newcommand{\LB}{\textnormal{LB}}
\newcommand{\UB}{\textnormal{UB}}
 \newcommand{\lb}{\textnormal{lb}}
 \newcommand{\ub}{\textnormal{ub}}

\begin{document}
\title{An Approximation Algorithm for a General Class of Multi-Parametric Optimization Problems\thanks{This work was supported by the DFG grants TH 1852/4-1 and RU 1524/6-1.}
}

\author{Stephan Helfrich\footnote{University of Kaiserslautern, Department of Mathematics, Paul-Ehrlich-Str.~14, D-67663~Kaiserslautern, Germany, email: \{helfrich,herzel,ruzika\}@mathematik.uni-kl.de}, Arne Herzel\footnotemark[\value{footnote}]~, Stefan Ruzika\footnotemark[\value{footnote}]~,	
	 Clemens Thielen\footnote{Weihenstephan-Triesdorf University of Applied Sciences, TUM Campus Straubing for Biotechnology and Sustainability, Am Essigberg~3, D-94315~Straubing, Germany, email: clemens.thielen@hswt.de
	}	
}


\maketitle

\begin{abstract}
	In a widely-studied class of multi-parametric optimization problems, the objective value of each solution is an affine function of real-valued parameters. Then, the goal is to provide an optimal solution set, i.e., a set containing an optimal solution for each non-parametric problem obtained by fixing a parameter vector. For many multi-parametric optimization problems, however, an optimal solution set of minimum cardinality can contain super-polynomially many solutions. Consequently, no polynomial-time exact algorithms can exist for these problems even if $\textsf{P}=\textsf{NP}$.
	
	We propose an approximation method that is applicable to a general class of multi-parametric optimization problems and outputs a set of solutions with cardinality polynomial in the instance size and the inverse of the approximation guarantee. This method lifts approximation algorithms for non-parametric optimization problems to their parametric version and provides an approximation guarantee that is arbitrarily close to the approximation guarantee of the approximation algorithm for the non-parametric problem. If the non-parametric problem can be solved exactly in polynomial time or if an FPTAS is available, our algorithm is an FPTAS.
	Further, we show that, for any given approximation guarantee, the minimum cardinality of an approximation set is, in general, not $\ell$-approximable for any natural number~$\ell$ less or equal to the number of parameters, and we
	discuss applications of our results to classical multi-parametric combinatorial optimizations problems.
	In particular, we obtain an FPTAS for the multi-parametric minimum $s$-$t$-cut problem, an FPTAS for the multi-parametric knapsack problem, as well as an approximation algorithm for the multi-parametric maximization of independence systems problem.\\

	
\end{abstract}

\section{Introduction}
Many optimization problems depend on parameters whose values are unknown or can only be estimated. Changes in the parameters may alter the set of optimal solutions or even affect feasibility of solutions. \emph{Multi-parametric optimization models} describe the dependencies of the objective function and/or the constraints on the values of the parameters. That is, for any possible combination of parameter values, multi-parametric optimization problems ask for an optimal solution and its objective value. 

In this article, we consider \emph{linear multi-parametric optimization problems} in which the objective depends affine-linearly on each parameter.  For simplicity, we focus on minimization problems, but all our reasoning and results can be applied to maximization problems as well.
Formally, for $K \in \N \setminus \{0\}$, (an instance of) a \emph{linear $K$-parametric optimization problem}~$\Pi$ is given by a nonempty (finite or infinite) \emph{set~$X$ of feasible solutions},  functions~$a, b_k : X \rightarrow \R$, $k = 1, \dots, K$, 
and a \emph{parameter set}~$\Lambda \subseteq \R^K$. Then,
the optimization problem is typically formulated (cf.~\cite{Oberdieck2016,Pistikopoulos2012}) as
\begin{align*}
\begin{Bmatrix}
\displaystyle\inf_{x\in X} f(x,\lambda) \colonequals a(x) + \sum_{k = 1}^K \lambda_k \cdot b_k(x)
\end{Bmatrix}_{\lambda \in \Lambda}.
\end{align*}
Fixing a parameter vector~$\lambda \in \Lambda$ yields (an instance of) the \emph{non-parametric version}~$\Pi(\lambda)$ of the linear $K$-parametric optimization problem. 
Moreover, the function $f: \Lambda \rightarrow \R \cup \{-\infty\}, \lambda \mapsto f(\lambda) \colonequals \inf_{x \in X} f(x, \lambda)$ that assigns the optimal objective value of~$\Pi(\lambda)$ to each parameter vector~$\lambda \in \Lambda$, is called the \emph{optimal cost curve}.	
%
The goal is to find a set~$S'\subseteq X$ of feasible solutions that contains an optimal solution for~$\Pi(\lambda)$ for each $\lambda \in \Lambda$ for which $\inf_{x \in X} f(x,\lambda)$ is attained.
Such a set~$S'$ is called an \emph{optimal solution set} of the multi-parametric problem and induces a decomposition of the parameter set~$\Lambda$: For each solution~$x \in S'$, the associated \emph{critical region}~$\Lambda(x)$ subsumes all parameter vectors~$\lambda\in\Lambda$ such that~$x$ is optimal for~$\Pi(\lambda)$.  

For many linear multi-parametric optimization problems, however, the \emph{cardinality of any optimal solution set can be} super-polynomially large, even if $K=1$~(see, for example, \cite{Allman:complexity2parameterMinCut,Carstensen:PHD,Gassner+Klinz:parametric-assignment,Nikolova+etal:ESA06,Ruhe:parametric-network-flows}). In general, this rules out per se the existence of polynomial-time exact algorithms even if $\textsf{P}=\textsf{NP}$. Approximation provides a concept to substantially reduce the number of required solutions while still obtaining provable solution quality. For the non-parametric version~$\Pi(\lambda)$, approximation is defined as follows (cf.~\cite{williamson_shmoys_2011:book_approximation}):
\begin{definition}\label{def:approximation}
	For $\beta \geq 1$ and a parameter vector $\lambda \in \Lambda$ such that $f(\lambda) \geq 0$, a feasible solution $x \in X$ is called \emph{$\beta$-approximate} (or a \emph{$\beta$-approximation}) for the non-parametric version~$\Pi(\lambda)$ if $f(x, \lambda) \leq \beta \cdot f(x',\lambda)$ for all $x' \in X$.
\end{definition}
This concept can be adapted to linear multi-parametric optimization problems. There, the task is then to find a set of solutions that contains a $\beta$-approximate solution for each non-parametric problem~$\Pi(\lambda)$.
Formally, this is captured in the following definition (cf.~\cite{Bazgan+etal:parametric,Giudici+etal:param-knapsack}):
\begin{definition}
	\label{def:approxsol}
	For $\beta \geq 1$, a finite set~$S \subseteq X$ is called a \emph{$\beta$-approximation set} for~$\Pi$ if it contains a $\beta$-approximate solution $x \in S$ for~$\Pi(\lambda)$ for any $\lambda \in \Lambda$ for which $f(\lambda)\geq0$. 
	An algorithm~$\A$ that computes a $\beta$-approximation set for any instance~$\Pi$ in time polynomially bounded in the instance size is called a \emph{$\beta$-approximation algorithm}.		
	A \emph{polynomial time approximation scheme} (PTAS) 
	is a family $(\A_{\varepsilon})_{\varepsilon>0}$ of algorithms such that, for every $\varepsilon >0$, algorithm $\A_{\varepsilon}$ is a $(1 + \varepsilon)$-approximation algorithm. A PTAS $(\A_{\varepsilon})_{\varepsilon>0}$ is a \emph{fully polynomial-time approximation scheme} (FPTAS) if the running time of $\A_\varepsilon$ is in addition polynomial in~$\frac{1}{\varepsilon}$.
\end{definition}

Next, we discuss some assumptions that are necessary in order to ensure a well-defined notion of approximation and to allow for the existence of efficient approximation algorithms. Note that the outlined (technical) assumptions are rather mild and they are satisfied for multi-parametric formulations of a large variety of well-known optimization problems. This includes well-known problems such as the knapsack problem, the minimum $s$-$t$-cut problem, and the maximization of independence systems problem (see Section~\ref{sec:Applications}), as well as the assignment problem, the minimum cost flow problem, the shortest path problem, and the metric traveling salesman problem (see Section~$5$ in~\cite{Bazgan+etal:parametric}).

Similar to the case of non-parametric problems, where non-negativity of the optimal objective value is required in order to define approximation (cf.~\cite{williamson_shmoys_2011:book_approximation} and Definition~\ref{def:approximation} above), approximation for multi-parametric problems can only be defined if the optimal objective value~$f(\lambda)$ is non-negative for any~$\lambda\in\Lambda$.
To ensure this, assumptions on the parameter set and the functions $a,b_k$, $k=1,\dots,K$, are necessary. An initial approach would be to assume nonnegativity of the parameter vectors as well as nonnegativity of the functions~$a,b_k$, $k=1,\dots,K$. A natural generalization also allows for negative parameter vectors. To this end, we consider a lower bound~$\lmin \in \R^K$ on the parameter set, i.e., $\Lambda\colonequals\bigtimes_{k = 1}^K [\lmin_k, \infty)$ (Assumption \ref{ass:poly}~\ref{ass:parameterset}). Then, assuming $f(x,\lmin)$ and $b_k(x)$, $k=1,\dots,K$, to be nonnegative for all~$x\in X$ (Assumption~\ref{ass:poly}~\ref{ass:polycomputablebounds}) guarantees nonnegativity of the optimal objective value for any~$\lambda\in\Lambda$.

Moreover, solutions must be polynomially encodable\footnote{This is a typical assumption in approximation (see, e.g.,~\cite{Papadimitriou+Yannakakis:multicrit-approx}). Nevertheless, our method can also be applied to problems where only the \enquote*{relevant solutions} can be encoded polynomially in the instance size. For example, not all feasible solutions of linear programs can be encoded polynomially in the instance size since they are implicitly determined by finitely many inequalities. However, it is sufficient to restrict to basic feasible solutions, which have encoding length polynomially bounded in the instance size~\cite{Groetschel1993}. 
} and the values~$a(x)$ and~$b_k(x)$, $k=1,\dots,K$, must be efficiently computable for any~$x \in X$ in order for the problem to admit any polynomial-time approximation algorithm. Hence, we assume that any solution~$x \in X$ is of polynomial encoding length and the values~$a(x)$ and $b_k(x)$, $k=1,\dots,K$, can be computed in time polynomial in the instance size and the encoding length of~$x$ (Assumption~\ref{ass:poly}~\ref{ass:polycomputable}). This implies that the values~$a(x)$ and $b_k(x)$, $k=1,\dots,K$, are rationals of polynomial encoding length. Consequently, the assumptions made so far imply the existence of positive rational bounds~$\LB$ and~$\UB$ such that $b_k(x), f(x,\lmin) \in \{0\} \cup [\LB, \UB]$ for all $x \in X$ and all $k = 1,\dots,K$.  
It is further assumed that $\LB$ and $\UB$ can be computed polynomially in the instance size (Assumption~\ref{ass:poly}~\ref{ass:polycomputablebounds}). Note that the numerical values of~$\LB$ and~$\UB$ may still be exponential in the instance size.

Extending the results for $1$-parametric optimization problems from~\cite{Bazgan+etal:parametric}, we study how an exact or approximate algorithm~$\ALG$ for the non-parametric version can be used in order to approximate the multi-parametric problem and which approximation guarantee can be achieved when relying on polynomially many calls to~$\ALG$. Hence, the last assumption is the existence of an exact algorithm or an approximation algorithm for the non-parametric version (Assumption~\ref{ass:poly}~\ref{ass:polyExistenceALG}). In summary, the following assumptions are made:

\begin{assumption}\label{ass:poly}~
	\begin{enumerate}[label=(\alph*), ref=(\alph*)]
	  \item\label{ass:parameterset} For some given $\lmin = (\lmin_1,\dots,\lmin_K)^\top \in \mathbb R^K$, the parameter set is of the form $\Lambda = \bigtimes_{k = 1}^K [{\lmin_k} , \infty)$.
	  \item\label{ass:polycomputable} Any $x \in X$ can be encoded by a number of bits polynomial in the instance size and the values~$a(x)$ and $b_k(x)$, $k =1,\dots,K$, can be computed in time polynomial in the instance size and the encoding length of~$x$. 
	  \item\label{ass:polycomputablebounds} Positive rational bounds $\LB$ and $\UB$ such that $b_k(x), f(x,\lmin) \in \{0\} \cup [\LB, \UB]$ for all $x \in X$ and all $k = 1, \dots, K$ can be computed in time polynomial in the instance size.  
			\item\label{ass:polyExistenceALG} For some~$\alpha \geq 1$, there exists an algorithm~$\ALG_{\alpha}$ that returns, for any parameter vector~$\lambda \in \Lambda$, a solution~$x'$ such that $f(x',\lambda) \leq \alpha \cdot f(x,\lambda)$ for all~$x \in X$.\footnote{The approximation guarantee~$\alpha$ is assumed to be independent of~$\lambda$. However, it is allowed that~$\alpha$ depends on the instance (such that the encoding length of~$\alpha$ is polynomially bounded in the encoding length of the instance). 
			} The running time is denoted by $T_{\ALG_{\alpha}}$.
		\end{enumerate}
\end{assumption}


\subsection{Related Literature}\label{sec:Literature}

Linear $1$-parametric problems are widely-studied in the literature. Under the assumption that there exists an optimal solution for any non-parametric version, the parameter set can be decomposed into critical regions consisting of finitely many intervals~$(-\infty, \lambda^1],[\lambda^1, \lambda^2], \dots , [\lambda^L, \infty)$ with the property that, for each interval, one feasible solution is optimal for all parameters within the interval. Assuming that~$L$ is chosen as small as possible, the parameter values~$\lambda^1,\dots, \lambda^L$ are exactly the points of slope change (the \emph{breakpoints}) of the piecewise-linear optimal cost curve. A general solution approach for obtaining the optimal cost curve is presented by Eisner and Severance~\cite{Eisner+Severance:method}. 
Exact solution methods for specific optimization problems exist for the linear $1$-parametric shortest path problem~\cite{Karp+Orlin:parametric-shortest-path}, the linear $1$-parametric assignment problem~\cite{Gassner+Klinz:parametric-assignment}, and the linear $1$-parametric knapsack problem~\cite{Eben-Chaime:par-knapsack}. Note that linear $1$-parametric optimization problems also appear in the context of some well-known combinatorial problems. For example, Karp and Orlin~\cite{Karp+Orlin:parametric-shortest-path} observe that the minimum mean cycle problem can be reduced to a linear $1$-parametric shortest path problem~\cite{Carstensen:PHD,Mulmuley+Shah:parametric-sp}, and Young et al.~\cite{Young+etal:parametric-sp} note that linear $1$-parametric programming
problems arise in the process of solving the minimum balance problem, the minimum concave-cost dynamic network flow problem~\cite{Graves+Orlin:concave-cost}, and matrix scaling~\cite{Orlin+Rothblum:optimal-scalings,Schneider+Schneider:matrix-scaling}.

These and many other problems share an inherent difficulty (see, e.g., Carstensen~\cite{Carstensen:PHD}): The optimal cost curve may have super-polynomially many breakpoints in general. This precludes the existence of polynomial-time exact algorithms even if $\textsf{P} = \textsf{NP}$. Nevertheless, there exist $1$-parametric optimization problems for which the number of breakpoints is polynomial in the instance size. For example, this is known for linear $1$-parametric minimum spanning tree problems~\cite{Fernandez-Baca+etal:SWAT96} as well as for special cases of linear $1$-parametric binary integer programs~\cite{Carstensen:PHD,Carstensen:parametric}, linear $1$-parametric maximum flow problems~\cite{Gallo+etal:parametric-flow,McCormick:parametric-max-flow}, and linear $1$-parametric shortest path problems~\cite{Erickson:SODA2010,Karp+Orlin:parametric-shortest-path,Young+etal:parametric-sp}.

Exact solution methods for general linear  multi-parametric optimization problems are studied by Gass and Saaty~\cite{Gass1955,Saaty1954} and Gal and Nedoma~\cite{Gal1972}. The minimum number of solutions needed to decompose the parameter set into critical regions, called the \emph{parametric complexity},\footnote{Also referred to as combinatorial facet complexity or facet complexity~\cite{Aissi2015}.} is a natural criterion to measure the complexity. As for the $1$-parametric case, the parametric complexity of a variety of problems is super-polynomial in the instance size. This even holds true for the special cases of minimum $s$-$t$-cut problems whose $1$-parametric versions are tractable~\cite{Allman:complexity2parameterMinCut}.
Known exceptions are certain linear $K$-parametric binary integer programs~\cite{Carstensen:PHD}, various linear $K$-parametric multiple alignment problems~\cite{FernandezBaca2000a}, linear $K$-parametric global minimum cut problems~\cite{Aissi2015,Karger2016}, and the linear $K$-parametric minimum spanning tree problem~\cite{Seipp:thesis}.

As outlined above, many linear multi-parametric optimization problems do not admit polynomial-time algorithms in general, even if $\textsf{P}=\textsf{NP}$ and $K = 1$. This fact strongly motivates the design of approximation algorithms for multi-parametric optimization problems. So far, approximation schemes exist only for linear $1$-parametric optimization problems. A general algorithm, which can be interpreted as an approximate version of the method of Eisner and Severance, is presented by Herzel et al.~\cite{Bazgan+etal:parametric}. 
The approximation of the linear $1$-parametric 0-1-knapsack problem is considered in~\cite{Giudici+etal:param-knapsack,Halman2018,Holzhauser2017}.

We conclude this section by expounding the relationship between multi-parametric optimization and multi-objective optimization. We first mention similarities and then discuss differences between (the approximation concepts for) both types of problems.
In a multi-objective optimization problem, the $K+1$~objective functions~$a,b_k$, $k=1,\dots,K$, are to be optimized over the feasible set~$X$ simultaneously, and a $\beta$-approximation set is a set~$S'\subseteq X$ of feasible solutions such that, for each solution $x \in X$, there exists a solution $x' \in S'$ that is at most a factor of~$\beta$ worse than~$x$ in each objective function~$a,b_k$, $k=1,\dots,K$. 
We refer to the seminal work of Papadimitriou and Yannakakis~\cite{Papadimitriou+Yannakakis:multicrit-approx} for further details.

When restricting to nonnegative parameter sets~$\Lambda$, linear multi-parametric problems can be solved exactly by methods that compute so-called (extreme) supported solutions of multi-objective problems. Moreover, since the functions~$a,b_k$, $k=1,\dots,K$, are linearly combined by a nonnegative parameter vector, multi-objective approximation sets are also multi-parametric approximation sets in this case. Surveys on exact methods and on the approximation of multi-objective optimization problems are provided by Ehrgott et al.~\cite{Ehrgott2016:surveyMOCO} and Herzel et al.~\cite{Herzel+etal:survey}, respectively. 
Using techniques from multi-objective optimization with the restriction that the functions $a,b_k$ are assumed to be strictly positive, multi-parametric optimization problems with nonnegative parameter sets are approximated in~\cite{Daskalakis:chord,Diakonikolas2011approximation,Diakonikolas2008succinct}. We note that the proposed concepts heavily rely on scaling of the objectives such that, for each solution~$x \in X$, all the pair-wise ratios of $a(x),b_k(x)$, $k=1,\dots,K$, are bounded by two. This clearly cannot be done if (strict subsets of) the function values of solutions~$x\in X$ are equal to zero.

Despite these connections, there are significant differences between the approximation of multi-pa\-ra\-metric and multi-objective problems: (1)~As already pointed out by Diakonikolas~\cite{Diakonikolas2011approximation},  the class of problems admitting an efficient multi-parametric approximation algorithm is larger than the class of problems admitting an efficient multi-objective approximation algorithm. For example, the multi-parametric minimum $s$-$t$-cut problem with positive parameter set can be approximated efficiently~\cite{Diakonikolas2011approximation}, whereas it is shown in~\cite{Papadimitriou+Yannakakis:multicrit-approx} that there is no FPTAS for constructing a multi-objective approximation set for the bi-objective minimum $s$-$t$-cut problem unless $\textsf{P}=\textsf{NP}$.
This is also highlighted by the simple fact that (2)~multi-objective approximation is not well-defined for negative objectives, whereas multi-parametric approximation allows the functions $a,b_k$ to be negative as long as the parameter set~$\Lambda$ is restricted such that $f(x,\lambda) \geq 0$ for all solutions~$x \in X$ and all parameter vectors~$\lambda\in \Lambda$. (3)~For nonnegative parameter sets, Herzel et al.\ \cite{Helfrich+etal:cones} show that, in the case of minimization, a multi-parametric $\beta$-approximation set is only a multi-objective $((K+1) \cdot \beta)$-approximation set. In the case of maximization, they even show that no multi-objective approximation guarantee can be achieved by a multi-parametric approximation in general.
(4)~There also exist substantial differences with respect to the minimum cardinality of approximation sets: For some $M \in \N$, consider a $2$-parametric maximization problem with feasible solutions $x^0, \dots, x^M$ such that $a(x^i) = \beta^{2i}$, $b_1(x^1) = \beta^{2(M-i)}$ for $i = 0,\dots, M$. Then, any multi-objective $\beta$-approximation set must contain all solutions, whereas $\{x^0,x^M\}$ is a multi-parametric $\beta$-approximation set.

Consequently, existing approximation algorithms for $1$-parametric and/or multi-objective optimization problems are not sufficient for obtaining efficient and broadly-applicable approximation methods for general multi-parametric optimization problems. This motivates the article at hand, in which we establish a theory of and provide an efficient method for the approximation of general linear multi-parametric problems.

\subsection{Our Contribution}
We provide a general approximation method for a large class of multi-parametric optimization problems by extending the ideas of both the approximation algorithm of Diakonikolas et al.~\cite{Diakonikolas2011approximation,Diakonikolas2008succinct} and the $1$-parametric approximation algorithm of Herzel et al.~\cite{Bazgan+etal:parametric} to linear multi-parametric problems. 
Note that, in~\cite{Bazgan+etal:parametric}, only $1$-parametric problems are considered, which leads to an easier structure of the optimal cost curve due to the one-dimensional parameter set, which allows for a bisection-based approximation algorithm. This bisection-based approach cannot be generalized to multi-dimensional parameter sets as considered here.

For any $0 < \varepsilon < 1$, we show that, if the non-parametric version can be appro\-xi\-mated within a factor of $\alpha \geq 1$, then the linear multi-parametric problem can be approximated within a factor of $(1 + \varepsilon) \cdot \alpha$ in running time polynomially bounded by the size of the instance and $\frac{1}{\varepsilon}$. That is, the algorithm outputs a set of solutions that, for any feasible vector of parameter values, contains a solution that $((1+ \varepsilon) \cdot \alpha)$-approximates all feasible solutions in the corresponding non-parametric problem.
Consequently, the availability of a polynomial-time exact algorithm or an (F)PTAS for the non-parametric problem implies the existence of an (F)PTAS for the multi-parametric problem.

In Section \ref{sec:ApproxScheme}, we show basic properties of the parameter set with respect to approximation. These results allow a decomposition of the parameter set by means of assigning each vector of parameter values to the approximating solution. We state our polynomial-time (multi-parametric) approximation method for the general class of linear multi-parametric optimization problems.

Furthermore, we discuss the task of finding a set of solutions with minimum cardinality that approximates the linear $K$-parametric optimization problem in Section~\ref{sec:minCard}. We adapt the impossibility result of~\cite{Diakonikolas2011approximation,Diakonikolas2008succinct}, which states that there does not exist an efficient approximation algorithm that provides any constant approximation factor on the minimum cardinality if the non-parametric problem can be approximated within a factor of $1+ \delta$ for some $\delta >0$. We extend this to the case that an exact non-parametric algorithm is available. Here, we show that there cannot exist an efficient approximation algorithm that yields an approximation set with cardinality less than or equal to $K$~times the minimum cardinality.

Section~\ref{sec:Applications} discusses applications of our general approximation algorithm to multi-parametric versions of several well-known optimization problems.
In particular, we obtain fully polynomial-time approximation schemes for the linear multi-parametric minimum $s$-$t$-cut problem 
and the multi-parametric knapsack problem (where approximation schemes for the linear $1$-parametric version have been presented in~\cite{Giudici+etal:param-knapsack,Holzhauser2017}). We also obtain an approximation algorithm for the multi-parametric maximization problem of independence systems, a class of problems where the well-known greedy method is an approximation algorithm for the non-parametric version.

\section{A General Approximation Algorithm}\label{sec:ApproxScheme}
We now present our approximation method for linear multi-parametric optimization problems satisfying Assumption~\ref{ass:poly}. We first sketch the general idea and then discuss the details.
In the following, given some $\beta \geq 1$, we simply say that~$x$ is a $\beta$-approximation for~$\lambda$ instead of~$x$ is a $\beta$-approximation for~$\Pi(\lambda)$ if this does not cause any confusion. Clearly, for each solution~$x \in X$, there is a (possibly empty) subset of parameter vectors $\Lambda' \subseteq \Lambda$ such that~$x$ is a $\beta$-approximation for all parameter vectors~$\lambda' \in \Lambda'$. 
Hence, the notion of $\beta$-approximation (sets) is relaxed as follows: A solution~$x$ is a $\beta$-approximation for $\Lambda' \subseteq \Lambda$ if it is a $\beta$-approximation for every $\lambda' \in \Lambda'$. Analogously, a set~$S \subseteq X$ is a $\beta$-approximation set for $\Lambda' \subseteq \Lambda$ if, for any $\lambda' \in \Lambda'$, there exists a solution~$x \in S$ that is a $\beta$-approximation for~$\lambda'$.

\bigskip

Let $0 < \varepsilon < 1$ be given and let $\alpha \geq 1$ be the approximation guarantee obtained by the algorithm~$\ALG_\alpha$ for the non-parametric version as in Assumption~\ref{ass:poly}~\ref{ass:polyExistenceALG}. The general idea of our multi-parametric approximation method can be described as follows: 
We show that there exists a compact subset $\Lapprox \subseteq \Lambda$ with the following property:
\begin{enumerate}[label=(\Alph*), ref=(\Alph*)]
	\item \label{property:LApprox} For each parameter vector $\lambda \in \Lambda \setminus \Lapprox$, there exists a parameter vector~$\lambda' \in \Lapprox$ such that any $((1 + \frac{\varepsilon}{2}) \cdot \alpha)$-approximation for~$\lambda'$ is also a $((1+ \frac{\varepsilon}{2}) \cdot \alpha + \frac{\varepsilon}{2})$-approximation for $\lambda$ (see Proposition~\ref{prop:WapproxForAugmentedProblem} and Corollary~\ref{cor:approxnormedweightset}). Thus, since $((1+ \frac{\varepsilon}{2}) \cdot \alpha + \frac{\varepsilon}{2}) \leq (1 + 2 \cdot \frac{\varepsilon}{2}) \cdot \alpha) = (1+\varepsilon)\cdot \alpha$, any $((1 + \frac{\varepsilon}{2}) \cdot \alpha)$-approximation for~$\lambda'$ is, in particular, a $((1+\varepsilon)\cdot \alpha)$-approximation for $\lambda$.
\end{enumerate}

Then, a grid $\Grid \subseteq \Lambda$ is constructed, where each~$\lambda' \in \Grid$ is computed as $\lambda'_k = \lmin_k + (1 + \frac{\varepsilon}{2})^{l_k}$ for some~$l_k \in \Z$ and $k = 1, \dots, K$, such that the following holds:
\begin{enumerate}[resume,label=(\Alph*), ref=(\Alph*)]
	\item \label{property:Cardinality} The cardinality of $\Grid$ is polynomially bounded in the encoding length of the instance and $\frac{1}{\varepsilon}$ (but exponential in~$K$), see~Proposition~\ref{lem:numberGridPoints}. 
	
	\item \label{property:Grid} For each parameter vector $\lambda' \in \Lapprox$, there exits a grid vector~$\bar{\lambda} \in \Grid$ such that 
	\begin{align*}
	    \bar{\lambda}_k - \lmin_k \leq \lambda'_k - \lmin_k \leq  \left(1 + \frac{\varepsilon}{2}\right) (\bar{\lambda}_k - \lmin_k) \quad\text{ for } k=1,\dots,K.
	\end{align*}
	Then, any $\alpha$-approximation for $\bar{\lambda}$ is a $((1 + \frac{\varepsilon}{2}) \cdot \alpha)$-approximation for $\lambda'$ (see Proposition~\ref{prop:apprximation regionLambda}). 
\end{enumerate}
It follows that, for each parameter vector~$\lambda \in \Lambda \setminus \Lapprox$, there exist a parameter vector $\bar{\lambda} \in \Lapprox$ and a grid vector~$\lambda' \in \Grid$ such that any $\alpha$-approximation for~$\lambda'$ is a $((1 + \frac{\varepsilon}{2}) \cdot \alpha)$-approximation for $\bar{\lambda}$ and, thus, a $((1 +  \varepsilon) \cdot \alpha)$-approximation for $\lambda$. Hence, algorithm~$\ALG_{\alpha}$ can be applied for the polynomially many parameter vectors in~$\Grid$, and collecting all solutions results in a $((1+ \varepsilon) \cdot \alpha)$-approximation set~$S$ for~$\Lambda$. 
This method yields a multi-parametric (F)PTAS if either a polynomial-time exact algorithm $\ALG_{1}$ or an (F)PTAS for the non-parametric version is available. Moreover, this approximation algorithm allows to easily assign  the corresponding approximate solution~$x \in S$ to each parameter vector~$\lambda \in \Lambda$.

\bigskip
We now present the details of the algorithm and start with Property~\ref{property:LApprox}. To this end, it is helpful to also allow parameter dependencies in the constant term. Hence, we define $F_0(x) \colonequals f(x,\lmin)$ and $F_k(x) \colonequals b_k(x)$, $k =1, \dots, K$, for all $x \in X$. Further, we let
$\R^{K+1}_\geqq \colonequals \{ w \in \R^{K+1} : w_i \geq 0, i = 0, \dots, K\}$ denote the $(K+1)$-dimensional nonnegative orthant.
Using this notation, the \emph{augmented multi-parametric problem} reads 
\begin{align}\label{eq:augmentedmultiparam}
\begin{Bmatrix}
\displaystyle\min_{ x \in X}  w_0 F_0(x) + w_1 F_1(x) + \dots + w_K F_K(x)
\end{Bmatrix}_{w \in \R^{K+1}_\geqq}
\end{align}
and the goal is to provide an optimal solution for any $w \in \R^{K+1}_\geqq$.
The vectors $w \in \R^{K+1}_\geqq$ are called \emph{weights}, and the  set of all weights is called the \emph{weight set} in order to distinguish it from the parameter set of the non-augmented problem. The terms $\beta$-approximate solution and $\beta$-approximation set for the augmented problem are defined analogously to Definitions~\ref{def:approximation} and~\ref{def:approxsol}, respectively. Note that the non-parametric version~$\Pi(\lambda)$ of~$\Pi$ for some $\lambda = (\lambda_1, \dots, \lambda_K) \in \Lambda$ coincides with the non-parametric version of the augmented problem for the weight~$w = (1, \lambda_1 - \lmin_1,\dots,\lambda_K - \lmin_K)$.

\smallskip

A solution~$x^*$ is optimal for some weight~$w \in \R^{K+1}_\geqq$ if and only if $x^*$ is optimal for $t \cdot w$ for any positive scalar $t>0$. An analogous result holds in the approximate sense: 
\begin{observation}\label{lem:multapproxinvariant}
	Let $x,x^* \in X$ be two feasible solutions. Then, for any positive scalar $t >0$ and $\beta \geq 1$, it holds that $\sum_{i = 0}^K w_i F_i(x^*) \leq \beta \cdot \sum_{i = 0}^K w_i F_i(x)$ if and only if
	$t \cdot \sum_{i = 0}^K w_i F_i(x^*) \leq t \cdot  \beta \cdot \sum_{i = 0}^K w_i F_i(x).$ 
\end{observation}
The conclusion of this observation is twofold: On the one hand, any $\beta$-approximation set for the augmented multi-parametric problem~\eqref{eq:augmentedmultiparam} is also a $\beta$-approximation set for~$\Pi$. On the other hand, restricting the weight set to the bounded $K$-dimensional simplex $W_1 \colonequals \{w \in \R^{K+1}_\geqq: \sum_{i = 0}^K w_i= 1\}$ again yields an equivalent problem. 

\bigskip

The compact set $\Lapprox \subseteq \Lambda$ satisfying Property~\ref{property:LApprox} can now be derived as follows: For $\beta \geq 1$ and $0 < \varepsilon' < 1$, a closed cone $\Wapprox \subseteq \R^{K+1}_{\geqq}$ is constructed such that any $\beta$-approximation set for $\Wapprox$ is a $(\beta + \varepsilon')$-approximation set for $\R^{K+1}_\geqq$, $W_1$, and~$\Pi$. 
Then, by Observation~\ref{lem:multapproxinvariant}, 
a $\beta$-approximation set for the intersection $\Wapproxone \colonequals \Wapprox \cap W_1$ is also a $(\beta + \varepsilon')$-approximation set for $\R^{K+1}_\geqq$, $W_1$, and~$\Pi$. Since~$W_1$ is compact, any closed subset of~$W_1$ is also compact. Thus, denoting the Minkowski sum of two sets~$A,B \subseteq \Lambda$ of parameter vectors by~$A+B$, the (continuous) function
\begin{align*}
\phi: W_1 \cap \{w:w_0 >0\} \rightarrow \Lambda, \quad (w_0,w_1, \dots, w_K) \mapsto \left(\frac{w_1}{w_0} + \lmin_1, \dots, \frac{w_K}{w_0} + \lmin_K\right)
\end{align*}
can be defined to obtain a compact subset~$\Lapprox \colonequals \phi(\Wapproxone)$ of~$\Lambda$. By choosing $\varepsilon' = \frac{\varepsilon}{2}$ and $\beta = (1 + \varepsilon')$, a compact subset $\Lapprox$ that satisfies Property~\ref{property:LApprox} is obtained.\footnote{Note that $\Lapprox$ could also be defined by means of the intersection $\Wapprox \cap \{w \in \R^{K+1}_\geq: w_0 = 1\}$. However, structural insights into the geometry of~$\Lapprox$ would then be missed. Moreover, the presented construction allows to easily derive lower and upper bounds on~$\Lapprox$, which are necessary for proving the polynomial bound on the cardinality of the grid~$\Grid$.}

\bigskip

The next results formalize this outline. Initially, an auxiliary result about convexity and approximation is given: For $\gamma \geq 1$, if a solution~$x$ is a $\gamma$-approximation for several weights $w^1, \dots, w^L \in \R^{K +1}_\geqq$, the same solution~$x$ is a $\gamma$-approximation for any weight in their convex hull.
\begin{lemma}\label{prop:alphaconvex}
	Let $\gamma \geq 1$ and a subset $W' \subseteq \R^{K+1}_\geqq$ be given. Then, any $\gamma$-approximation~$x \in X$ for~$W'$ is also a $\gamma$-approximation for the convex hull~$\conv(W')$.
\end{lemma}
\begin{proof}
Let~$w$ be some weight in the convex hull of~$W'$. Then, $w = \sum_{l=1}^L \theta_l w^l$ for some $L \in \N$, $w^1, \dots, w^L \in W'$, and $\theta_1, \dots, \theta_L \in [0,1]$ with $\sum_{l=1}^L \theta_l = 1$. Thus, for any $x ' \in X$,
\begin{align*}
 \sum_{i = 0}^K w_i F_i(x) &= \sum_{i = 0}^K \sum_{l=1}^L \theta_l \cdot w^l_i F_i(x) = \sum_{l=1}^L \theta_l \cdot \sum_{i = 0}^K w^l_i F_i(x) \\
&\leq \sum_{l=1}^L \theta_l \cdot \gamma \cdot \sum_{i = 0}^K w^l_i F_i(x')  = \gamma \cdot \sum_{i = 0}^K w^l_i F_i(x'),
\end{align*}  
which implies that $x$ is a $\gamma$-approximation for $w$.
\end{proof}

The next results establish the compact set $\Lapprox \subseteq \Lambda$ of parameter vectors such that any $\beta$-approximation set for $\Lapprox$ is a $(\beta + \varepsilon')$-approximation set for $\Lambda$.

Let~$w$ be a strictly positive weight whose components~$w_i$ for~$i$ in some index set $\emptyset \neq I \subseteq \{0,\dots, K\}$ sum up to a small threshold. The next proposition states that, instead of computing an approximate solution for~$w$, one can compute an approximate solution for the weight obtained by projecting all component~$w_i$, $i \in I$, to zero, and still obtain a \enquote*{sufficiently good} approximation guarantee for~$w$.

To this end, for a set~$I \subseteq \{0, \dots, K\}$ of parameter indices, the projection~$\proj^I:\R^{K+1} \rightarrow \R^{K+1}$ that maps all components~$w_i$ of a vector~$w \in \R^{K+1}$ with indices~$i\in I$ to zero is defined by 
\begin{align*}
\proj^{I}_i (w) \colonequals \begin{cases}
0, &\text { if } i \in I,\\
w_i, &\text { else}.\end{cases} 
\end{align*}
\begin{lemma}\label{prop:ApproximationBoundaryAugmenetedProblem}
	Let $0 < \varepsilon' < 1$ and $\beta \geq 1$. Further, let  $\emptyset \neq I \subsetneq \{0,\dots,K\}$ be an index set and let $w \in \R^{K+1}_\geqq$ be a weight for which
	\begin{align*}
	\sum_{i \in I} w_i = \frac{\varepsilon' \cdot \LB}{\beta \cdot \UB} \cdot \min_{j \notin I} w_j.
	\end{align*}
	Then, any $\beta$-approximation for $w$ is a $(\beta + \varepsilon')$-approximation for $\proj^{I}(w)$.
\end{lemma}
\begin{proof}	
Let $x \in X$ be a $\beta$-approximate solution for $w$. We have to show that, for any~$x' \in X$,
\begin{align*}
\sum_{j \notin I} w_j F_j(x) \leq (\beta + \varepsilon') \cdot \sum_{j \notin I} w_j F_j(x').
\end{align*}
Since~$x$ is a $\beta$-approximation for~$w$, we know that, for any solution~$x' \in X$,
\begin{align*}
	\sum_{i \in I} w_i F_i(x) + \sum_{j \notin I} w_j F_j(x) &\leq \beta \cdot \left( \sum_{i \in I} w_i F_i(x') + \sum_{j \notin I} w_j F_j(x') \right),
\end{align*}
which implies that
\begin{align*}
&&\sum_{j \notin I} w_j F_j(x) - \beta \sum_{j \notin I} w_j F_j(x') &\leq \sum_{i \in I} w_i \cdot \left( \beta \cdot F_i(x') - F_i(x) \right) \\
&& &\leq \sum_{i \in I} w_i \cdot \beta \cdot \UB 
= \varepsilon' \cdot \LB \cdot \min_{j \notin I} w_j.	
\end{align*}
Note that, for any solution $x'' \in X$, it holds that
\begin{align*}
	\sum_{j \notin I} w_j F_j(x'') \in \{0\} \cup \left[   \LB \cdot \min_{j \notin I}  w_j, \sum_{j \notin I} w_j \cdot \UB \right].
\end{align*}
If $\sum_{j \notin I} w_j F_j(x') = 0$, it holds that
	\begin{align*}
		\sum_{j \notin I} w_j F_j(x) < \LB \cdot \min_{i \notin I} w_i
	\end{align*} 
	and, therefore, $\sum_{j \notin I} w_j F_j(x) = 0$. Hence, in this case, we have
\begin{align*}
	\sum_{j \notin I} w_j F_j(x) = 0 = (\beta + \varepsilon') \cdot  \sum_{j \notin I} w_j F_j(x')
\end{align*} 
If $\sum_{j \notin I} w_j F_j(x') \geq \LB \cdot \min_{j \notin I} w_j$, it holds that
\begin{align*}
	\sum_{j \notin I} w_j F_j(x)
	&\leq \varepsilon' \cdot \LB \cdot \min_{j \notin I} w_j + \beta \sum_{j \notin I} w_j F_j(x')\\
	&\leq \varepsilon' \sum_{j \notin I} w_j F_j(x') + \beta \sum_{j \notin I} w_j F_j(x') = (\beta + \varepsilon') \cdot \sum_{j \notin I} w_j F_j(x'),
\end{align*} 
which proves the claim.  
\end{proof}
Let $w \in \R^{K+1}_\geqq$ and $\emptyset \neq I \subsetneq \{0,\dots,K\}$ be given as in Lemma~\ref{prop:ApproximationBoundaryAugmenetedProblem}. By the convexity property from Lemma~\ref{prop:alphaconvex}, every $\beta$-approximation for $w$ is not only a $(\beta + \varepsilon')$-approximation for $\proj^{I}(w)$, but also a $(\beta + \varepsilon')$-approximation for all weights in $\conv(\{w,\proj^{I}(w)\})$. This suggests the following definition:
\begin{definition}\label{def:sets}
Given $0 <  \varepsilon' < 1$ and $\beta \geq 1$, the threshold used in the proof of Lemma~\ref{prop:ApproximationBoundaryAugmenetedProblem} is denoted by
\begin{align}\label{eq:c}
	c \colonequals \frac{\varepsilon' \cdot \LB}{\beta \cdot \UB} \in (0,1).
\end{align}
	Additionally, for any index set $\emptyset \neq I \subsetneq \{0,\dots,K\}$, we define
\begin{align}\label{eq:Psets}
	P_{<} (I) \colonequals \left\{ w \in \R^{K+1}_\geqq: \sum_{i \in I} w_i < c \cdot w_j \text{ for all } j \notin I \right\}.
\end{align}
The set~$P_\leq (I)$ is defined analogously by replacing ``$<$'' by ``$\leq$'' in~\eqref{eq:Psets}. Note that~$P_\leq(I)$ is a polyhedron. Moreover, we define
\begin{align*}
	P_= (I) \colonequals P_\leq(I) \setminus P_<(I)
\end{align*} and, finally,
\begin{align*}
\Wapprox \colonequals \R^{K+1}_\geqq \setminus \left( \bigcup\limits_{\emptyset \neq I \subsetneq \{0,\dots,K\}}   P_<(I) \right).
\end{align*}
\end{definition}
\noindent
Figure~\ref{fig:Wapprox} provides a visualization of the sets defined in Definition~\ref{def:sets}.
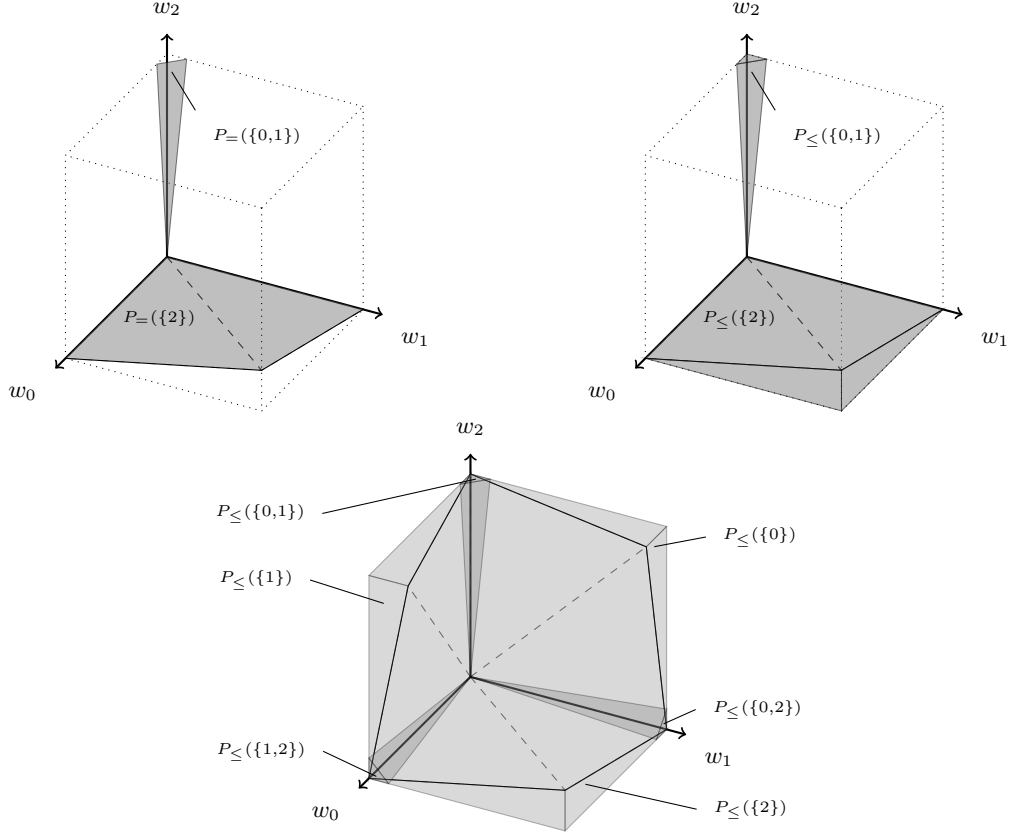
\begin{figure}[H]	
\begin{subfigure}{0.49\textwidth}
\centering
	\begin{tikzpicture}[scale = 0.27,x  = {(-0.5cm,-0.5cm)},
	y  = {(0.9659cm,-0.25882cm)},
	z  = {(0cm,1cm)}]
	
	\def\a{10}
	\def\eps{2}
	\def\opac{0.5}
	\def\color{black!50}
	
	\draw[thick,->] (0,0,0) -- (11,0,0) node[label=below left:$w_0$] {};
	\draw[thick,->] (0,0,0) -- (0,11,0) node[label=below right:$w_1$] {};
	\draw[thick,->] (0,0,0) -- (0,0,11) node[label=$w_2$] {};
	
	\draw[dashed, fill =\color, opacity =\opac] (0,0,0) -- (\a,\a,\eps) -- (\a,0,0) -- cycle;
	\draw[dashed, fill = \color, opacity =\opac] (0,0,0) -- (\a,\a,\eps) -- (0,\a,0) -- cycle;
	\draw[\color]  (\a,0,0) --  (\a,\a,\eps) -- (0,\a,0) ;
	
	\draw[dotted] (\a,\a,\a) -- (\a,\a,0);
	\draw[dotted] (\a,0,0) -- (\a,\a,0) -- (0,\a,0);
	
	\draw[dotted] (\a,\a,\a) -- (\a,0,\a);
	\draw[dotted] (\a,0,0) -- (\a,0,\a) -- (0,0,\a);
	
	\draw[dotted] (\a,\a,\a) -- (0,\a,\a);
	\draw[dotted] (0,\a,0) -- (0,\a,\a) -- (0,0,\a);
	
	\draw[\color,fill =\color, opacity =\opac] (0,0,0) -- (0.5*\eps,0,\a) -- (0,0.5*\eps,\a) -- cycle;	
	
	\draw[thin]  (8,1,2.5) node[label={below right:$\scriptstyle P_=(\{2\})$}] {}; 
	\draw[thin]  (0.5,0.5,9.5) --  (3,3,9.5) node[label={below right:$\scriptstyle P_=(\{0,1\})$}] {};

	\end{tikzpicture}
\end{subfigure}	
\begin{subfigure}{0.49\textwidth}
	\centering
	\begin{tikzpicture}[scale = 0.27,x  = {(-0.5cm,-0.5cm)},
	y  = {(0.9659cm,-0.25882cm)},
	z  = {(0cm,1cm)}]
	
	\def\a{10}
	\def\eps{2}
	\def\opac{0.5}
	\def\color{black!50}
	
	\draw[thick,->] (0,0,0) -- (11,0,0) node[label=below left:$w_0$] {};
	\draw[thick,->] (0,0,0) -- (0,11,0) node[label=below right:$w_1$] {};
	\draw[thick,->] (0,0,0) -- (0,0,11) node[label=$w_2$] {};
	
	\draw[dotted] (\a,\a,\a) -- (\a,\a,0);
	\draw[dotted] (\a,0,0) -- (\a,\a,0) -- (0,\a,0);
	
	\draw[dotted] (\a,\a,\a) -- (\a,0,\a);
	\draw[dotted] (\a,0,0) -- (\a,0,\a) -- (0,0,\a);
	
	\draw[dotted] (\a,\a,\a) -- (0,\a,\a);
	\draw[dotted] (0,\a,0) -- (0,\a,\a) -- (0,0,\a);
	
	\draw[dashed, fill =\color, opacity =\opac] (0,0,0) -- (\a,\a,\eps) -- (\a,0,0) -- cycle;
	\draw[dashed, fill = \color, opacity =\opac] (0,0,0) -- (\a,\a,\eps) -- (0,\a,0) -- cycle;
	\draw[\color]  (\a,0,0) --  (\a,\a,\eps) -- (0,\a,0) ;
	
	\draw[\color, fill =\color, opacity =\opac] (\a,\a,\eps) -- (\a,\a,0) -- (\a,0,0) -- cycle;
	\draw[\color, fill =\color, opacity =\opac] (\a,\a,\eps) -- (\a,\a,0) -- (0,\a,0) -- cycle;
	
	\draw[\color,fill =\color, opacity =\opac] (0,0,0) -- (0.5*\eps,0,\a) -- (0,0.5*\eps,\a) -- cycle;	
	\draw[\color, fill =\color, opacity =\opac] (0,0.5*\eps,\a) -- (0,0,\a) -- (0.5*\eps,0,\a) -- cycle;
	
	\draw[thin]  (8,1,2.5) node[label={below right:$\scriptstyle P_\leq(\{2\})$}] {}; 
	\draw[thin]  (0.5,0.5,9.5) --  (3,3,9.5) node[label={below right:$\scriptstyle P_\leq(\{0,1\})$}] {};
	\end{tikzpicture}
\end{subfigure}	
\begin{subfigure}{\textwidth}
	\centering
	\begin{tikzpicture}[scale = 0.27,x  = {(-0.5cm,-0.5cm)},
	y  = {(0.9659cm,-0.25882cm)},
	z  = {(0cm,1cm)}]
	
	\def\a{10}
	\def\eps{2}
	\def\opac{0.3}
	\def\color{black!50}
	
	\draw[thick,->] (0,0,0) -- (11,0,0) node[label=below left:$w_0$] {};
	\draw[thick,->] (0,0,0) -- (0,11,0) node[label=below right:$w_1$] {};
	\draw[thick,->] (0,0,0) -- (0,0,11) node[label=$w_2$] {};
	
	\draw[dashed, fill =\color, opacity =\opac] (0,0,0) -- (\a,\a,\eps) -- (\a,0,0) -- cycle;
	\draw[dashed, fill = \color, opacity =\opac] (0,0,0) -- (\a,\a,\eps) -- (0,\a,0) -- cycle;
	\draw[\color]  (\a,0,0) --  (\a,\a,\eps) -- (0,\a,0) ;
	\draw[\color, fill =\color, opacity =\opac] (\a,\a,\eps) -- (\a,\a,0) -- (\a,0,0) -- cycle;
	\draw[\color, fill =\color, opacity =\opac] (\a,\a,\eps) -- (\a,\a,0) -- (0,\a,0) -- cycle;
	
	\draw[dashed,\color,fill = \color, opacity =\opac] (0,0,0) -- (\a,\eps,\a) -- (\a,0,0) -- cycle;
	\draw[dashed,\color,fill = \color, opacity =\opac] (0,0,0) -- (\a,\eps,\a) -- (0,0,\a) -- cycle;
	\draw[\color]  (\a,0,0) --  (\a,\eps,\a) -- (0,0,\a) ;
	\draw[\color, fill =\color, opacity =\opac] (\a,\eps,\a) -- (\a,0,\a) -- (\a,0,0) -- cycle;
	\draw[\color, fill =\color, opacity =\opac] (\a,\eps,\a) -- (\a,0,\a) -- (0,0,\a) -- cycle;
	
	\draw[dashed,\color,fill = \color, opacity =\opac] (0,0,0) -- (\eps,\a,\a) -- (0,0,\a) -- cycle;
	\draw[dashed,\color,fill = \color, opacity =\opac] (0,0,0) -- (\eps,\a,\a) -- (0,\a,0) -- cycle;
	\draw[\color] (0,0,\a) -- (\eps,\a,\a) -- (0,\a,0);
	\draw[\color, fill =\color, opacity =\opac] (\eps,\a,\a) -- (0,\a,\a) -- (0,0,\a) -- cycle;
	\draw[\color, fill =\color, opacity =\opac] (\eps,\a,\a) -- (0,\a,\a) -- (0,\a,0) -- cycle;

%
%

	\draw[\color, fill =\color, opacity =\opac] (0,0,0) -- (\a,0,0.5*\eps) -- (\a,0.5*\eps,0) -- cycle;	
	\draw[\color, fill =\color, opacity =\opac] (\a,0.5*\eps,0) -- (\a,0,0) -- (\a,0,0.5*\eps) -- cycle;
	\draw[\color, fill =\color, opacity =\opac] (0,0,0) -- (0,\a,0.5*\eps) -- (0.5*\eps,\a,0) -- cycle;
	\draw[\color, fill =\color, opacity =\opac] (0,\a,0.5*\eps) -- (0,\a,0) -- (0.5*\eps,\a,0) -- cycle;
%
	
	\draw[\color, fill =\color, opacity =\opac] (0,0,0) -- (0.5*\eps,0,\a) -- (0,0.5*\eps,\a) -- cycle;	
	\draw[\color, fill =\color, opacity =\opac] (0,0.5*\eps,\a) -- (0,0,\a) -- (0.5*\eps,0,\a) -- cycle;

	%
	
	
	
	\draw[thin]  (0.1,9.9,0.3) --  (1,12,2) node[label={right:$\scriptstyle P_\leq(\{0,2\})$}] {};
	\draw[thin]  (9.85,0.3,0.1) --  (12,-1.3,2) node[label={left:$\scriptstyle P_\leq(\{1,2\})$}] {};
	\draw[thin]  (0.05,0.3,9.85) --  (3,-6,8) node[label={left:$\scriptstyle P_\leq(\{0,1\})$}] {};

	\draw[thin]  (8,10,1.3) --  (4,13.6,-1) node[label={right:$\scriptstyle P_\leq(\{2\})$}] {}; 
	\draw[thin]  (0.25,9.5,9) --  (0,12,10) node[label={right:$\scriptstyle P_\leq(\{0\})$}] {};
	\draw[thin]  (9.5,0.5,8.5) --  (10,-3,9) node[label={left:$\scriptstyle P_\leq(\{1\})$}] {};
	\end{tikzpicture}
\end{subfigure}	
	\caption{Illustration of the sets~$P_=(I)$, $P_\leq(I)$, and~$\Wapprox$ for a linear multi-parametric problem with~$K =2$.
		Top left: Visualization of $P_{=}(\{0,1\})$ and $P_{=}(\{2\})$.
		Top right: Visualization of the corresponding full-dimensional sets $P_\leq(\{0,1\})$ and $P_\leq(\{2\})$.
		Bottom:	Visualization of all sets $P_{\leq}(I)$. The set~$\Wapprox$ is the complement of the union of the sets $P_{<}(I) = P_{\leq}(I) \setminus P_{=}(I)$ for $\emptyset \neq I \subsetneq \{0,\dots, K\}$. Note that none of the visualized sets are bounded from above.}
	\label{fig:Wapprox}
\end{figure}

Let $\bar{w} \in \Wapprox$ such that $\bar{w} \in P_=(I)$ for some index set $\emptyset \neq I \subsetneq \{0,\dots,K\}$. Lemma~\ref{prop:ApproximationBoundaryAugmenetedProblem} implies that a $\beta$-approximation for $\bar{w}$ is a $(\beta + \varepsilon')$-approximation for $\conv( \bar{w}, \proj^I(\bar{w}))$. However, the reverse statement is needed: For $w \in \R^{K+1}_\geqq \setminus \Wapprox$, does there exist a weight $\bar{w} \in \Wapprox$ such that a $\beta$-approximation for $\bar{w}$ is a $(\beta + \varepsilon')$-approximation for $w$? Proposition~\ref{prop:WapproxForAugmentedProblem} will show that this holds true. In fact, the corresponding proof is constructive and relies on the lifting procedure described in the following.
 
\bigskip
 
Consider some weight $w \in \R^{K+1}_\geqq \setminus \Wapprox$, i.e., $w \in P_\leq(I)$ for some index set $\emptyset \neq I \subsetneq \{0,\dots,K\}$. Instead of computing an approximate solution for $w$, a $\beta$-approximation for the corresponding \emph{lifted weight}~$\bar{w} \in P_=(I)$ (satisfying $w \in \conv(\{\bar{w}, \proj^I(\bar{w})\})$) can be computed, which is then a $(\beta + \varepsilon')$-approximation for $w$. The next lemma formalizes the lifting.

\begin{lemma}\label{lem:Lifting}
	Let $\emptyset \neq I \subsetneq \{0,\dots,K\}$ be an index set and let $w \in P_<(I)$. Define
	\begin{align*}
	\bar{w}_i \colonequals 
	\begin{cases}
		\frac{w_i}{\sum_{j \in I} w_j} \cdot c \cdot \min_{j \notin I} w_j, &\text{ if } i \in I \text{ and } \sum_{j \in I} w_j >0,\\
		\frac{1}{|I|} \cdot c \cdot \min_{j \notin I} w_j, &\text{ if } i \in I \text{ and } w_j =0 \text{ for all } j \in I,\\
		w_i, &\text{ if } i \notin I.
	\end{cases}
	\end{align*}
	Then, $\bar{w} \in P_=(I)$ and $w \in \conv(\{\bar{w}, \proj^I(\bar{w})\})$. In particular, $\bar{w}_i \geq w_i$ for all $i \in I$.
\end{lemma}
\begin{proof}
	First consider the case that $w_j = 0$ for all~$j \in I$. \textcolor{blue}{Then} it holds that
	\begin{align*}
	\sum_{i \in I} \bar{w}_i = \sum_{i \in I} \frac{1}{|I|} \cdot c \cdot \min_{j \notin I} w_j = c \cdot \min_{j \notin I} w_j = c \cdot \min_{j \notin I} \bar{w}_j,
	\end{align*}
	which yields that $\bar{w} \in P_=(I)$. Moreover, $w = \proj^I(\bar{w}) \in \conv(\{\bar{w}, \proj^I(\bar{w})\})$.
	Now consider the case that $w_j \neq 0$ for some~$j \in I$. Here, it holds that
	\begin{align*}
		\sum_{i\in I} \bar{w}_i = \sum_{i \in I} \frac{w_i}{\sum_{j \in I} w_j} \cdot c \cdot \min_{j \notin I} w_j = c \cdot \min_{j \notin I} w_j = c \cdot \min_{j \notin I} \bar{w}_j,
	\end{align*}
	which again yields that $\bar{w} \in P_=(I)$. Note that, since $w \in P_<(I)$, we must have $c \cdot \displaystyle\min_{j \notin I} w_j > \sum_{j \in I} w_j \geq 0$. Thus, the weight~$w$ can be written as a convex combination of~$\bar{w}$ and $\proj^I(\bar{w})$ by
	\begin{align*}
	w = \frac{\sum_{j \in I} w_j}{c \cdot \min_{j \notin I} w_j} \cdot \bar{w} + \left( 1 - \frac{\sum_{j \in I} w_j}{c \cdot \min_{j \notin I} w_j} \right) \cdot \proj^I(\bar{w}),
	\end{align*}
	which concludes the proof.
\end{proof}
When given a weight $w \in P_<(I)$ for some index set~$I$, a lifted weight~$\bar{w} \in P_=(I)$ can be constructed using Lemma~\ref{lem:Lifting}. A $\beta$-approximation for $\bar{w}$ is then a $(\beta + \varepsilon')$-approximation for $w$ due to Lemma~\ref{prop:alphaconvex} and Lemma~\ref{prop:ApproximationBoundaryAugmenetedProblem}. Next, it is shown that this idea generalizes to the set~$\Wapprox$ in the following way: For each weight $w \notin \Wapprox$, a weight~$\bar{w} \in \Wapprox$ can be found such that any $\beta$-approximation for~$\bar{w}$ is a $(\beta + \varepsilon')$-approximation for~$w$. The remaining task is to prove that this holds true for weights contained in $P_<(I) \cap P_<(I')$ for two (or more) different index sets~$I$ and~$I'$, since using the previous construction for~$I$ might result in a lifted weight that is still contained in~$P_<(I')$ and vice versa. Notwithstanding, such weights can inductively be lifted with respect to different index sets and, if this is done in a particular order, a weight is obtained that is contained in~$\Wapprox$ after at most~$K$ lifting steps.

The following lemma states that, for a weight~$w$ that is not contained in~$P_<(I)$ for some index set~$I$, increasing any of its components~$w_i$ with indices~$i \in I$ preserves the fact that the weight is not contained in~$P_<(I)$.

\begin{lemma}\label{lem:liftinvariance}
	Let $w \in \R^{K+1}_\geqq \setminus P_<(I)$ for some index set $\emptyset \neq I \subsetneq \{0,\dots,K\}$. Let $\bar{w} \in \R^{K+1}_\geqq$ be a weight such that $\bar{w}_i \geq w_i$ for all $i \in I$ and $\bar{w}_j = w_j$ for all $j \notin I$. Then, $\bar{w} \notin P_<(I)$.
\end{lemma}
\begin{proof}
	Since $w \in P_<(I)$, it holds that
	\begin{align*}
	\sum_{i \in I} \bar{w}_i \geq \sum_{i \in I}w_i \geq c \cdot \min_{j \notin I} w_j = c \cdot \min_{j \notin I} \bar{w}_j,
	\end{align*}
	which proves the claim.
\end{proof}
\noindent
Now, we can prove the central result for Property~\ref{property:LApprox}. Note that the proof is constructive.
\begin{proposition}\label{prop:WapproxForAugmentedProblem}
	Let $0 < \varepsilon' < 1$ and $\beta \geq 1$ be given. Then, for any weight~$w \in \R^{K+1}_\geqq \setminus \Wapprox$, there exists a weight $\bar{w} \in \Wapprox$ such that any $\beta$-approximation for $\bar{w}$ is a $(\beta + \varepsilon')$-approximation for $w$.
\end{proposition}
\begin{proof}
	Let $w \in \R^{K+1}_\geqq$. Without loss of generality,  assume that $w_0 \leq w_1 \leq \dots \leq w_K$ holds (otherwise, the ordering of the indices can be changed due to symmetry of $\Wapprox$). First, it is shown that, in this case, all index sets $I$ such that $w \in P_<(I)$ are of the form $I =\{0,\dots,k\}$ for some $k \in \{0,\dots,K-1\}$: Let $i \in I$ and $j \notin I$ for some index set $\emptyset \neq I \subsetneq \{0,\dots,K\}$ for which $w \in P_<(I)$. Then,
	\begin{align*}
	w_i \leq \sum_{i' \in I} w_{i'} < c \cdot \min_{j' \notin I} w_{j'} \leq w_j,
	\end{align*}
	which implies that $i < j$ and, thus, $I = \{0,\dots,k\}$ for some $k \in \{0,\dots,K-1\}$. To shorten the notation, we use the abbreviations $[K] \colonequals \{0, \dots, K\}$ and $[\bar{k}] \colonequals \{0, \dots, \bar{k}\}$ for $\bar{k} \in [K]$ in the remainder of this proof.\\
	
	Since $w \notin \Wapprox$, we have $w \in P_<(I)$ for at least one index set $\emptyset \neq I \subsetneq [K]$. Hence, we choose $k^{\max} \in [K-1]$ to be the largest index such that $w \in P_<([k^{\max}])$ holds. Similarly, choose $k^0 \in [K-1]$ to be the smallest index such that $w \in P_<([k^0])$ holds. This means that $w \notin P_<([k])$ for all $k \in [K-1]$ with $0 \leq k < k^0$ and $k^{\max} < k < K$. Further, set $w^0 \colonequals w$ and construct a (finite) sequence $w^0,w^1, \dots, w^L$ of weights and a corresponding sequence $k^0< k^1 < \dots < k^{L-1}$ of indices such that, for each $\ell \in\{1, \dots, L\}$, the following statements hold:
	\begin{enumerate}[label=(\alph*), ref=(\alph*)]
		\item\label{prop:WapproxForAugmentedProblema} $w^{\ell}_i \geq w^{\ell-1}_i$ for $i=0, \dots, k^{\ell-1}$ and $w^{\ell}_j = w^{\ell-1}_j$ for $j = k^{\ell-1} +1 ,\dots, K$.
			
		\item\label{prop:WapproxForAugmentedProblemb} $0 < w^\ell_0 \leq w^{\ell}_1 \leq \dots \leq w^{\ell}_K$.
		
		\item\label{prop:WapproxForAugmentedProblemc} $w^{\ell -1} \notin P_<([k])$ for $k \in [K-1]$ with $0\leq k < k^{\ell-1}$ or $k^{\max} < k < K$.
		
		\item\label{prop:WapproxForAugmentedProblemd} $w^{\ell} \in P_=([k^{\ell'}])$ for $\ell' = 0,\dots, \ell-1$.
		
		\item\label{prop:WapproxForAugmentedProbleme} $w \in \conv(\{w^{\ell}\} \cup \{\proj^{[k^0]}(w^\ell), \dots, \proj^{[k^{\ell -1}]}(w^{\ell})\})$.
	\end{enumerate}
	The construction, which is illustrated in Figure~\ref{fig:ConstrIndexSetsAbdWeights}, is as follows:
	Given a weight $w^\ell \in \R^{K+1}_\geqq \setminus \Wapprox$ with $w^{\ell}_0 \leq w^\ell_1 \leq \dots \leq w^\ell_K$, we set~$k^{\ell}$ to be the smallest index such that $w^\ell \in P_<([k^\ell])$ and, analogously to Lemma~\ref{lem:Lifting}, define
	\begin{align*}
	w^{\ell +1}_i \colonequals 
	\begin{cases}
	\frac{w^{\ell}_i}{\sum_{j = 0}^{k^\ell} w^\ell_j} \cdot c \cdot w^{\ell}_{k^{\ell} +1} &\text{for } i = 0,\dots, k^{\ell}  \text{ if } \sum_{j=0}^{k^{\ell}} w^{\ell}_j > 0, \\
	\frac{1}{k^{\ell} + 1} \cdot c \cdot w^{\ell}_{k^{\ell} + 1} &\text{for } i = 0,\dots, k^{\ell} \text{ otherwise},\\
	w^{\ell}_i &\text{for } i =k^{\ell} + 1, \dots, K.	\end{cases}
	\end{align*}
	We repeat this construction until, for some~$L \in \N$, the weight~$w^L$ is not contained in~$P_<([k])$ for any $k \in [K-1]$. Note that Statement~\ref{prop:WapproxForAugmentedProblemb} implies that, for any~$\ell \in \N$, the weight~$w^\ell$ cannot be contained in~$P_<(I)$ for any index set~$I$ that is not of the form $I = [k]$ for some $k \in [K-1]$. Moreover, Statement~\ref{prop:WapproxForAugmentedProblemc} implies that $k^{\max} \geq k^\ell$ and that, for $\ell \geq1$ and $k \in[k^{\ell-1} - 1]$, it holds that $\sum_{i = 0}^k w^{\ell -1}_i \geq c \cdot w^{\ell -1}_{k+1}$. Therefore, if $\sum_{j=0}^{k^{\ell-1}} w^{\ell -1}_j >0$, it holds that
	\begin{align*}
	\sum_{i = 0}^k w^\ell_i = \sum_{i = 0}^k \frac{w^{\ell -1}_i}{\sum_{j=0}^{k^{\ell-1}} w^{\ell -1}_j  } \cdot c \cdot w^{\ell -1}_{k^{\ell}} 
	\geq \frac{c \cdot w^{\ell -1}_{k+1}}{\sum_{j=0}^{k^{\ell-1}}  w^{\ell -1}_j  } \cdot c \cdot w^{\ell -1}_{k^{\ell}} = c \cdot w^{\ell}_{k+1}.
	\end{align*}
	 Similarly, if $w^{\ell -1}_j = 0$ for all $j \in [k^{\ell-1}]$, it holds that
	\begin{align*}
	\sum_{i = 0}^k w^\ell_i = \sum_{i = 0}^k \frac{1}{k^{\ell -1} + 1} \cdot c \cdot w^{\ell-1}_{k^{\ell-1}+1} = (k+1) \cdot w^{\ell}_{k+1} > c \cdot w^{\ell}_{k+1}.
	\end{align*}
	Thus, in both cases, we obtain that $w^{\ell}\notin P_<([k])$ for $k = 0,\dots, k^{\ell-1}-1$.	
	Since Statement~\ref{prop:WapproxForAugmentedProblemd} implies that $w^{\ell} \notin P_<([k^{\ell-1}])$, this yields that $k^{\ell} > k^{\ell-1}$ (for $\ell \geq1$) and, hence, the construction indeed terminates after at most $k^{\max} - k^0<K$ steps with~$w^L \in \Wapprox$.
	
	Furthermore, Statement~\ref{prop:WapproxForAugmentedProblemd} and Lemma~\ref{prop:ApproximationBoundaryAugmenetedProblem} imply that any $\beta$-approximation for~$w^L$ is a $(\beta + \varepsilon')$-approximation for $\proj^{[k^{\ell'}]}(w^L)$ for each $l' \in 0,\dots, L-1$ and, thus, also for~$w$ using Statement~\ref{prop:WapproxForAugmentedProbleme} and the convexity Lemma~\ref{prop:alphaconvex}.\\
	
	It remains to show that Statements \ref{prop:WapproxForAugmentedProblema}--\ref{prop:WapproxForAugmentedProbleme} hold for each $\ell\in\{1,\dots,L\}$. Statement~\ref{prop:WapproxForAugmentedProblema} holds due Lemma~\ref{lem:Lifting}. Statements~\ref{prop:WapproxForAugmentedProblemb}--\ref{prop:WapproxForAugmentedProbleme} are proven by induction over $\ell$:
	
	For $\ell = 1$, in order to prove Statement~\ref{prop:WapproxForAugmentedProblemb}, first consider the case that $w^0_0 > 0$. In this case, $w^1_0>0$ by Statement~\ref{prop:WapproxForAugmentedProblema}. Next, consider the case that $w^0_0 =\dots, w^0_k = 0$ and $w^0_{k+1} >0$ for some $k \in [K-1]$ (note that $w$ cannot be the zero vector since $w \notin \Wapprox$). In this case, we must have $k = k^0$ by definition of~$k^0$, and, therefore,
	\begin{align*}
	w^1_0 = \frac{1}{k^0 +1} \cdot c \cdot w^0_{k^0 +1} = \frac{1}{k+1} \cdot c \cdot w^0_{k+1} > 0.
	\end{align*}
	The inequality $w^1_{k^0} \leq w^1_{k^0 + 1}$ even holds with strict inequality since, in both cases, it holds that $w^1_{k^0} \leq c \cdot w^0_{k^0 +1} < w^0_{k^0 +1} = w^1_{k^0 + 1}$. All other inequalities of Statement~\ref{prop:WapproxForAugmentedProblemb} follow from the corresponding inequalities for $\ell = 0$ (or trivially hold for $i = 1,\dots,k^1$ if $w^0_0 = \dots = w^0_{k^0} = 0$). Statement~\ref{prop:WapproxForAugmentedProblemc} is a direct consequence of our choice of~$k^0$ and~$k^{\max}$, and Statements~\ref{prop:WapproxForAugmentedProblemd} and \ref{prop:WapproxForAugmentedProbleme} immediately follow from Lemma~\ref{lem:Lifting}.
	
	Now assume that Statements~\ref{prop:WapproxForAugmentedProblemb}--\ref{prop:WapproxForAugmentedProbleme} hold for some $\ell \in \{1,\dots, L-1\}$. Then, Statements~\ref{prop:WapproxForAugmentedProblemb}--\ref{prop:WapproxForAugmentedProbleme} hold for $\ell +1$: The inequality $w^{\ell +1}_0 > 0$ holds since $w^{\ell +1}_0 \geq w^{\ell}_0 >0$ due to Statements~\ref{prop:WapproxForAugmentedProblema} and~\ref{prop:WapproxForAugmentedProblemb}. Again, the inequality $w^{\ell +1}_{k^\ell} \leq w^{\ell +1}_{k^\ell+1}$ holds with strict inequality since
	\begin{align*}
	w^{\ell +1}_{k^\ell} = \frac{w^\ell_{k^\ell}}{\sum_{j=0}^{k^\ell} w^\ell_j} \cdot c \cdot w^{\ell}_{k^\ell +1} \leq c \cdot w^\ell_{k^\ell + 1} < w^\ell_{k^\ell +1} = w^{\ell +1}_{k^\ell + 1},
	\end{align*}
	and all other inequalities of Statement~\ref{prop:WapproxForAugmentedProblemb} immediately follow from the corresponding inequalities for~$\ell$. In order to prove Statement~\ref{prop:WapproxForAugmentedProblemc}, note that, for $k \in [k^\ell -1]$, it holds that $w^\ell \notin P_<([k])$ by the choice of~$k^\ell$. For $k = k^{\max} +1 , \dots, K$, we have $w^\ell \notin P_<([k])$ by Statement~\ref{prop:WapproxForAugmentedProblema} and Lemma~\ref{lem:liftinvariance}. For Statement~\ref{prop:WapproxForAugmentedProblemd}, we have $w^\ell \in P_=([k^{\ell'}])$, i.e.,
	\begin{align*}
		\sum_{i = 0}^{k^{\ell'}} w^\ell_i = c \cdot w^\ell_{k^{\ell'} + 1}
	\end{align*}
	for $\ell' = 0, \dots, \ell -1$. Thus,
	\begin{align*}
		\sum_{i = 0}^{k^{\ell'}} w^{\ell + 1}_i = \sum_{i = 0}^{k^{\ell'}} \frac{w^\ell_i}{\sum_{j=0}^{k^\ell} w^\ell_j} \cdot c\cdot w^\ell_{k^\ell + 1} = \frac{c \cdot w^\ell_{k^{\ell'} +1}}{\sum_{j=0}^{k^\ell} w^\ell_j} \cdot c \cdot w^{\ell}_{k^\ell + 1} = c \cdot w^{\ell+1}_{k^{\ell'} +1 }
	\end{align*}
	for $\ell' = 0,\dots, \ell -1$. Moreover, by Lemma~\ref{lem:Lifting}, it holds that $w^{\ell +1} \in P_=([k^\ell])$, which concludes the proof of Statement~\ref{prop:WapproxForAugmentedProblemd}. Finally, Statement~\ref{prop:WapproxForAugmentedProbleme} holds for~$\ell +1$ since, by induction hypothesis, we know that $w \in \conv\left( \{w^\ell\} \cup \{ \proj^{[k^0]}(w^\ell), \dots, \proj^{[k^{\ell -1}]}(w^\ell)\} \right)$, which means that there exist coefficients $\theta_0,\dots, \theta_\ell \in [0,1]$ such that
	\begin{align*}
		w = \theta_{\ell} \cdot w^\ell + \sum_{\ell' =0}^{\ell -1} \theta_{\ell'} \cdot \proj^{[k^{\ell'}]}(w^\ell) \text{ and } \sum_{\ell' = 0}^{\ell} \theta_{\ell'} = 1.
	\end{align*}
	Lemma~\ref{lem:Lifting} implies that $w^\ell \in \conv\left( \{w^{\ell +1},\proj^{[k^{\ell}]}(w^{\ell +1})\}\right)$, i.e., there exists some $\mu \in [0,1]$ such that
	\begin{align*}
		w^\ell = \mu \cdot w^{\ell +1} + (1 - \mu) \cdot \proj^{[k^{\ell}]}(w^{\ell +1}).
	\end{align*}
	Note that, since $[k^{\ell'}] \subseteq [k^{\ell}]$ for $\ell' = 0, \dots, \ell-1$, it holds that
	\begin{align*}
	\proj^{[k^{\ell'}]} \left( \proj^{[k^{\ell}]}(w^{\ell +1})\right) = \proj^{[k^{\ell}]}(w^{\ell +1})
	\end{align*}
	and, thus,
	\begin{align*}
	w &= \theta_{\ell} \cdot w^{\ell} +  \sum_{\ell' = 0}^{\ell -1} \theta_{\ell'} \cdot \proj^{[k^{\ell'}]} (w^{\ell})\\
	&= \theta_{\ell} \cdot \left( \mu \cdot w^{\ell +1} + (1 - \mu) \cdot \proj^{[k^{\ell}]}(w^{\ell +1})  \right) + \sum_{\ell' = 0}^{\ell -1} \theta_{\ell'} \cdot \proj^{[k^{\ell'}]} \left(  \mu \cdot w^{\ell+1} + (1 - \mu) \cdot \proj^{[k^{\ell}]}(w^{\ell+1})  \right)\\
	&= \theta_{\ell} \cdot \mu  \cdot  w^{\ell+1} + \theta_{\ell} \cdot (1 - \mu) \cdot \proj^{[k^{\ell}]}(w^{\ell +1})\\
	&\hspace{3cm} +\sum_{\ell' = 0}^{\ell-1} \theta_{\ell'} \cdot \mu \cdot \proj^{[k^{\ell'}]}( w^{\ell +1} )  +\sum_{\ell' = 0}^{\ell-1} \theta_{\ell'} \cdot (1 - \mu) \cdot  \proj^{[k^{\ell'}]}(\proj^{[k^{\ell}]}(w^{\ell +1}))\\
	&=\theta_{\ell} \cdot \mu  \cdot  w^{\ell+1}  + \sum_{\ell' = 0}^{\ell-1} \theta_{\ell'} \cdot \mu \cdot \proj^{[k^{\ell'}]}( w^{\ell +1} ) +
	 (1 - \mu) \cdot \proj^{[k^{\ell}]}(w^{\ell+1})
	\end{align*}
	with
	\begin{align*}
	 \theta_{\ell} \cdot \mu  + \sum_{\ell' = 0}^{\ell-1} \theta_{\ell'} \cdot \mu +	 (1 - \mu) = \mu + (1 - \mu) = 1,
	\end{align*}
	i.e., $w \in \conv\left( \{w^{\ell+1} \} \cup \{ \proj^{[k^0]}(w^{\ell+1}), \dots, \proj^{[k^{\ell}]}(w^{\ell+1}) \}\right)$, which completes the induction and the proof.
\end{proof}

\begin{figure}[H]
	\centering
	\begin{tikzpicture}
	\begin{scope}[name = Voll, scale = 0.3,x  = {(-0.5cm,-0.5cm)},
	y  = {(0.9659cm,-0.25882cm)},
	z  = {(0cm,1cm)}]
	
	\def\a{10}
	\def\eps{2}
	\def\opac{0.3}
	\def\color{black!50}
	
	\draw[thick,->] (0,0,0) -- (11,0,0) node[label = below:$w_0$] {};
	\draw[thick,->] (0,0,0) -- (0,11,0) node[label = below:$w_1$] {};
	\draw[thick,->] (0,0,0) -- (0,0,11) node[label = right:$w_2$] {};
	
	\draw[\color,fill =\color, opacity =\opac] (0,0,0) -- (\a,\a,\eps) -- (\a,0,0) -- cycle;
	\draw[\color,fill = \color, opacity =\opac] (0,0,0) -- (\a,\a,\eps) -- (0,\a,0) -- cycle;
	
	\draw[\color,fill = \color, opacity =\opac] (0,0,0) -- (\a,\eps,\a) -- (\a,0,0) -- cycle;
	\draw[\color,fill = \color, opacity =\opac] (0,0,0) -- (\a,\eps,\a) -- (0,0,\a) -- cycle;
	
	\draw[\color,fill = \color, opacity =\opac] (0,0,0) -- (\eps,\a,\a) -- (0,0,\a) -- cycle;
	\draw[\color,fill = \color, opacity =\opac] (0,0,0) -- (\eps,\a,\a) -- (0,\a,0) -- cycle;
	
	\draw[dotted] (\a,\a,\eps) -- (\a,\a,0);
	\draw[dotted] (\a,0,0) -- (\a,\a,0) -- (0,\a,0);
	
	\draw[dotted] (\a,\eps,\a) -- (\a,0,\a);
	\draw[dotted] (\a,0,0) -- (\a,0,\a) -- (0,0,\a);
	
	\draw[dotted] (\eps,\a,\a) -- (0,\a,\a);
	\draw[dotted] (0,\a,0) -- (0,\a,\a) -- (0,0,\a);

	\draw[\color,fill =\color, opacity =\opac] (0,0,0) -- (\a,0,0.5*\eps) -- (\a,0.5*\eps,0) -- cycle;
	\draw[\color,fill =\color, opacity =\opac] (0,0,0) -- (0,\a,0.5*\eps) -- (0.5*\eps,\a,0) -- cycle;
	\draw[\color,fill =\color, opacity =\opac] (0,0,0) -- (0.5*\eps,0,\a) -- (0,0.5*\eps,\a) -- cycle;	
	%
	\draw[blue,thick,opacity = 0.2,fill=blue,dashed] (0,0,5) -- (11,0,5) -- (11,11,5)  -- (0,11,5) -- cycle;
	
	\end{scope}
	\draw[->,bend left] (1,1) to [out=30, in = 150]  (4.8,1.9);
	\begin{scope}[shift ={(5,-1.5)} , scale = 0.7, transform shape]
	\draw[<->] (0,6)  -- (0,0)  -- (6,0) ;
	\node[label=below:$w_0$] at (6,0) {};
	\node[] at (0,0) {};
	\node[label=left:$w_1$] at (0,6) {}; 
	\draw[fill = black!50, opacity = 0.4] (0,0) -- (0,5) -- (5,0) -- (0,0);
	
	
	
	\draw (2.5,2.5) -- (3,2.8) node[label ={[above right,,xshift=0cm, yshift=-0.2cm]:{$P_=(\{0,1\})$}}] {};
	\draw (2,16/3) -- (2.2,16/3) node[label ={[right,,xshift=0cm, yshift=-0.2cm]:{$P_=(\{0\})$}}] {};
	\draw (16/3,2) -- (16/3,2.2) node[label ={[above,,xshift=0cm, yshift=-0.2cm]:{$P_=(\{1\})$}}] {};
	
	\draw[fill = black!50, opacity = 0.2] (0,0) -- (5.5,16.5/8) --  (5.5,0);
	\draw[fill = black!50, opacity = 0.2] (0,0) -- (16.5/8,5.5) --  (0,5.5);
	
	\draw[->] (1,1)   -- (-1,-1) node[label=below:$w_2$]{};
	
	\draw[dashed,fill = red, opacity = 0.1] (1.35,3.65) -- (0,3.65) -- (0,0);	
	
	\node[red, label ={[red,xshift=0cm, yshift=-0.2cm]above: $w^0$}] at (0.2,2) {$\times$};
	\node[red, label ={[red,xshift=-0.2cm, yshift=0cm]right: $w^1$}] at (0.78,2) {$\times$};
	\node[red, label ={[red,xshift=-0.2cm, yshift=0cm]right: $w^2$}] at (1.35,3.65) {$\times$};
	
	\node[red, label ={[red,xshift=0.2cm, yshift=0cm]left: $\proj^{\{0\}} (w^2)$}] at (0,3.65) {$\times$};
	\node[red, label ={[red,xshift=0.2cm, yshift=0cm]left: $\proj^{\{0,1\}}(w^2)$}] at (0,0) {$\times$};
	\draw[dotted,red] (0,2) -- (0.78,2);
	\draw[dotted,red] (0,0) -- (1.35,3.65);
	\end{scope}
	\end{tikzpicture}
	\caption{Illustration of the sequence of (lifted) weights constructed in the proof of Proposition~\ref{prop:WapproxForAugmentedProblem} for a weight $w \notin \Wapprox$ with $w_0 < w_1 < w_2 = 1$. Left: Embedding of $\{ w \in \R^3_\geqq: w_2 = 1 \}$ into $\R^3_\geqq$. Right: Cross-section at $w_2 = 1$. Note that $w^2 \in \Wapprox$ and $w = w^0 \in \conv\left( \{w^2,\proj^{\{0\}}(w^2),\proj^{\{0,1\}}(w^2)\} \right)$.}
	\label{fig:ConstrIndexSetsAbdWeights}
\end{figure}
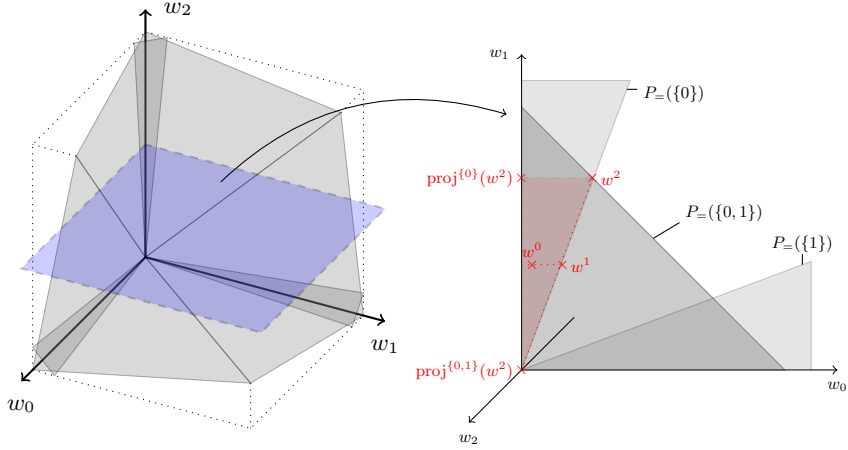

The following corollary states that the same result holds true for the $K$-dimensional simplex $W_1 = \{ w \in \R^{K+1}_\geqq: \sum_{i = 0}^K w_i = 1\}$, see Figure~\ref{fig:Wapprox1} for an illustration.

\begin{corollary}\label{cor:approxnormedweightset}
	For $0 < \varepsilon' <1$ and $\beta \geq 1$, define
	\begin{align*}
	\Wapproxone \colonequals \Wapprox \cap W_1.
	\end{align*}
	For each weight $w \in W_1 \setminus \Wapproxone$, there exists a weight $w' \in \Wapproxone$ such that any $\beta$-approximation for $w'$ is a $(\beta + \varepsilon')$-approximation for $w$.
\end{corollary}
\begin{proof}
	Note that $\Wapprox$ is a cone, i.e., $w \in \Wapprox$ if and only if $t \cdot w \in \Wapprox$ for each $t >0$. In particular, for each weight~$w \in \R^{K+1}_\geqq\setminus \{0\}$, it holds that
	\begin{align*}
		w \in \Wapprox \Longleftrightarrow \frac{1}{\sum_{i=0}^K w_i} \cdot w \in \Wapproxone.
	\end{align*}
	Thus, the claim follows immediately from Observation~\ref{lem:multapproxinvariant} and Lemma~\ref{prop:ApproximationBoundaryAugmenetedProblem}.
\end{proof}
\begin{figure}[H]
\centering
\begin{tikzpicture}\begin{scope}[scale = 0.3, x = {(-0.5cm,-0.5cm)}, y = {(0.9659cm,-0.25882cm)}, z = {(0cm,1cm)}]
\def\a{10}
\def\eps{2}
\def\opac{0.3}
\def\color{black!50}

\draw[thick,->] (0,0,0) -- (11,0,0) node[label=below left:$w_0$] {};
\draw[thick,->] (0,0,0) -- (0,11,0) node[label=below right:$w_1$] {};
\draw[thick,->] (0,0,0) -- (0,0,11) node[label=$w_2$] {};

\draw[\color,fill =\color, opacity =\opac] (0,0,0) -- (\a,\a,\eps) -- (\a,0,0) -- cycle;
\draw[\color,fill = \color, opacity =\opac] (0,0,0) -- (\a,\a,\eps) -- (0,\a,0) -- cycle;

\draw[\color,fill = \color, opacity =\opac] (0,0,0) -- (\a,\eps,\a) -- (\a,0,0) -- cycle;
\draw[\color,fill = \color, opacity =\opac] (0,0,0) -- (\a,\eps,\a) -- (0,0,\a) -- cycle;

\draw[\color,fill = \color, opacity =\opac] (0,0,0) -- (\eps,\a,\a) -- (0,0,\a) -- cycle;
\draw[\color,fill = \color, opacity =\opac] (0,0,0) -- (\eps,\a,\a) -- (0,\a,0) -- cycle;

\draw[dotted] (\a,\a,\eps) -- (\a,\a,0);
\draw[dotted] (\a,0,0) -- (\a,\a,0) -- (0,\a,0);

\draw[dotted] (\a,\eps,\a) -- (\a,0,\a);
\draw[dotted] (\a,0,0) -- (\a,0,\a) -- (0,0,\a);

\draw[dotted] (\eps,\a,\a) -- (0,\a,\a);
\draw[dotted] (0,\a,0) -- (0,\a,\a) -- (0,0,\a);

\draw[\color,fill =\color, opacity =\opac] (0,0,0) -- (\a,0,0.5*\eps) -- (\a,0.5*\eps,0) -- cycle;
\draw[\color,fill =\color, opacity =\opac] (0,0,0) -- (0,\a,0.5*\eps) -- (0.5*\eps,\a,0) -- cycle;
\draw[\color,fill =\color, opacity =\opac] (0,0,0) -- (0.5*\eps,0,\a) -- (0,0.5*\eps,\a) -- cycle;

\draw[blue!30,fill = blue!30,opacity=0.5](5,0,0) -- (0,0,5) --  (0,5,0) -- cycle;
\draw[white,fill = white,scale = 5,opacity=0.5] (0.0083,0.0826, 1 - 0.0083 -0.0826) -- (0.0826, 0.0083, 1 -0.0826 - 0.0083 )  -- (0.4762, 0.0476, 1 - 0.4762 - 0.0476) -- (0.9091, 0.0083, 1 - 0.9091 - 0.0083 ) -- (0.9091, 0.0826, 1 -0.9091 - 0.0826 ) -- (0.4762, 0.4762, 1 - 0.4762 - 0.4762) --  (0.0826, 0.9091, 1 - 0.0826 - 0.9091) -- (0.0083, 0.9091, 1 - 0.0083 - 0.9091 ) -- (0.0476, 0.4762, 1 - 0.0476 - 0.4762 ) -- cycle;

\end{scope}
\begin{scope}[ shift = {(4.5,-1.5)}, scale = 4.5,x  = {(0cm,0cm)},	y  = {(1cm,0cm)},	z  = {(0.5cm,0.8660cm)}]

\draw[black!30,fill = black!30](1.02,-0.01,-0.01) -- (-0.01,-0.01,1.02) --  (-0.01,1.02,-0.01) -- cycle;
\draw[white,fill = white] (0.0083,0.0826, 1 - 0.0083 -0.0826) -- (0.0826, 0.0083, 1 -0.0826 - 0.0083 )  -- (0.4762, 0.0476, 1 - 0.4762 - 0.0476) -- (0.9091, 0.0083, 1 - 0.9091 - 0.0083 ) -- (0.9091, 0.0826, 1 -0.9091 - 0.0826 ) -- (0.4762, 0.4762, 1 - 0.4762 - 0.4762) --  (0.0826, 0.9091, 1 - 0.0826 - 0.9091) -- (0.0083, 0.9091, 1 - 0.0083 - 0.9091 ) -- (0.0476, 0.4762, 1 - 0.0476 - 0.4762 ) -- cycle;

\def\t{2}
\draw[white,fill = white] (0.0083,0.0826, 1 - 0.0083 -0.0826) -- (0.0826, 0.0083, 1 -0.0826 - 0.0083 )  -- (0.4762, 0.0476, 1 - 0.4762 - 0.0476) -- (0.9091, 0.0083, 1 - 0.9091 - 0.0083 ) -- (0.9091, 0.0826, 1 -0.9091 - 0.0826 ) -- (0.4762, 0.4762, 1 - 0.4762 - 0.4762) --  (0.0826, 0.9091, 1 - 0.0826 - 0.9091) -- (0.0083, 0.9091, 1 - 0.0083 - 0.9091 ) -- (0.0476, 0.4762, 1 - 0.0476 - 0.4762 ) -- cycle;

\draw[fill = blue,opacity = 0.1](1.02,-0.01,-0.01) -- (-0.01,-0.01,1.02) --  (-0.01,1.02,-0.01) -- cycle;

\node at (0.33,0.33, 0.33) {$\Wapproxone$};
\draw[dashed] (0.0083,0.0083, 1 - 0.0083 -0.0083) -- (1 -0.0166, 0.0083, 0.0083) -- (0.0083, 1- 0.0166, 0.0083) -- cycle;

\node[label=below left:{$\scriptstyle (1,0,0)^\top$}] at (1,0,0) {};
\node[label=below right:{$\scriptstyle (0,1,0)^\top$}] at (0,1.0) {};
\node[label=above:{$\scriptstyle (0,0,1)^\top$}] at (0,0,1) {};
\end{scope}
\end{tikzpicture}
\caption{Illustration of the set $\Wapproxone$ for a linear multi-parametric optimization problem ($K = 2$). 
	Left: $\Wapproxone$ as a subset of $\R^3_\geqq$. Right: Schematic view of $\Wapproxone$ (light blue). The dashed lines indicate the boundary of $\bWapproxone$ defined in Lemma~\ref{lem:bWapprox1}.}
\label{fig:Wapprox1}
\end{figure}
The following lemma provides a lower bound on the components of weights~$w \in \Wapproxone$. This allows us to derive lower and upper bounds on $\Lapprox$, which will be useful when proving the polynomial cardinality of the grid~$\Grid$.
\begin{lemma}\label{lem:bWapprox1}
	Let $0 < \varepsilon' < 1$ and $\beta \geq 1$. Define
	\begin{align*}
	\bWapproxone \colonequals \left\{ w \in W_1 : w_i \geq \frac{1}{(K+1)!} \cdot c^K \text{ for all } i \in \{0,\dots,K\} \right\}.
	\end{align*}
	Then, $\Wapproxone \subseteq \bWapproxone \subseteq W_1$.
\end{lemma}
\begin{proof}
	Let $w \in \Wapproxone$. By symmetry of $W_1$, $\Wapproxone$, and $\bWapproxone$, we can assume without loss of generality that $w_0 \leq w_1 \leq \dots  \leq w_K$ holds. Since $w \in \Wapprox$, $w$ satisfies $w \notin P_{<}(I)$ for all $ \emptyset \neq I \subsetneq \{0,\dots,K\}$. In particular, this holds for all $I = \{0,\dots,k\}$ with $0 \leq k \leq K-1$. Hence,
	\begin{align*}
	    w_0 \geq c \cdot w_1, \quad  2 w_1 \geq w_0 + w_1 \geq c\cdot w_2,\ \dots \ ,K \cdot  w_{K-1} \geq \sum_{i=0}^{K-1} w_i \geq c \cdot w_K.
	\end{align*}
	With $(K +1) \cdot w_K \geq \sum_{i = 0}^K w_i = 1$, it follows that $w_i \geq \frac{1}{ ( K +1)!}\cdot c^K$ for all $i \in \{0,\dots, K\}$.
\end{proof}
\noindent
Next, $\Wapproxone$ is transformed to $\Lapprox$, see Figure~\ref{fig:Lapprox} for an illustration.
Recall that $$\phi: W_1 \cap \{ w: w_0 > 0\} \longrightarrow \Lambda, (w_0,w_1, \dots,w_K) \mapsto \left(\frac{w_1}{w_0} + \lmin_1, \dots, \frac{w_K}{w_0} + \lmin_K\right).$$
\begin{corollary}\label{cor:LapproxApproximationGuarantee}
	For $ 0 < \varepsilon' <1$ and $\beta \geq 1$, define $\Lapprox \colonequals \phi(\Wapproxone)$.
	Then, for each parameter vector~$\lambda \in \Lambda \setminus \Lapprox$, there exists a parameter vector~$\lambda' \in \Lapprox$ such that any $\beta$-approximation for~$\lambda'$ is a $(\beta+\varepsilon')$-approximation for~$\lambda$. 
\end{corollary}
\begin{proof}
Let $\lambda \in \Lambda \setminus \Lapprox$. Define $w=(w_0,\dots,w_k)$ by 
\begin{align*}
w_0 \colonequals \frac{1}{1 +\sum_{k=1}^K (\lambda_k - \lmin_k)},&& \text{ and } && w_i \colonequals \frac{\lambda_i - \lmin_i}{1 +\sum_{k=1}^K \lambda_k - \lmin_k} \text{ for } i = 1, \dots, K.
\end{align*}
Then, $w \in W_1$ and, thus, there exists a weight $w' \in \Wapproxone$ such that any $\beta$-approximation for~$w'$ is a $(\beta+\varepsilon')$-approximation for~$w$ by Corollary~\ref{cor:approxnormedweightset}. Observation~\ref{lem:multapproxinvariant} implies that any
$\beta$-approximation for $\phi(w') \in \Lapprox$ is a $\beta$-approximation for $w'$, which in turn is a		
$(\beta+\varepsilon')$-approximation for~$w$. Applying Observation~\ref{lem:liftinvariance} again yields that any $(\beta + \varepsilon')$-approximation for~$w$ is also a $(\beta+\varepsilon')$-approximation for~$\lambda = \phi(w)$.  	
\end{proof}
\begin{figure}[H]
\centering
\begin{tikzpicture}
\begin{scope}[scale = 0.4,domain=0:8]

\draw[black!25, fill = black!25, opacity = 1] (1,10.5)  -- (1,1) -- (10.5,1) -- (10.5, 10.5) -- (1,10.5);

\draw[black!25, fill = black!25, opacity = 1] (8,1) -- (10.5, 1.27) -- (10.5,0) -- (2,0) -- (2,0.25);
\draw[black!25, fill = black!25, opacity = 1] (1,8) -- (1.27, 10.5) -- (0, 10.5) -- (0,2) -- (0.25, 2);

\draw[black!25, fill = black!25, opacity = 1] (0.125,0.0625) -- (0.5,0.25) -- (2,0.25) -- (2,0) -- (0,0) -- (0,2) -- (0.25,2) -- (0.25,0.5) -- (0.0625,0.125) -- cycle;

\draw[very thin, dashed] (0.7,8.3) -- (8.3,0.7) -- (0.05,0.05) -- cycle;
\node at (1.5,0.55) {$\Lapprox$};

\draw[<->] (0,11) node[left] {$\lambda_2$} -- (0,0) -- (11,0) node[below] {$\lambda_1$};
\end{scope}
\begin{scope}[shift = {(7.5,0.3)},scale = 0.7]
\draw[] (0,6) -- (0,0) -- (6,0);

\draw[ black!25,fill = black!25, opacity = 1] (0,0) -- (0,10/7) -- (4/7, 10/7)-- (10/7, 4/7) --  (10/7,0) -- (0,0);

\node[label =below left:$\lmin$] at (0,0) {$\times$};
\draw[ black!25,fill = black!25, opacity = 1] (0,0) -- (5,2) -- (6, 2) -- (6,0);
\draw[ black!25,fill = black!25, opacity = 1] (0,0) -- (2,5) -- (2,6) -- (0,6);
\node[label =below:$\lmin_1 + \frac{\varepsilon'}{\beta} \frac{\LB}{\UB}$] at (2,0) {$\times$};
\node[label =left:$\lmin_2 + \frac{\varepsilon'}{\beta} \frac{\LB}{\UB}$] at (0,2) {$\times$};

\draw[dotted] (0,2) -- (5,2) -- (5,0) node[label = below:$\lmin_1 + 1$] {$\times$};
\draw[dotted] (2, 0) -- (2,5) -- (0,5) node[label = left:$\lmin_2 + 1$] {$\times$};

\node at (5,5) {$\Lapprox$};

\draw[dashed] (1.02,6) -- (0.47,0.47) -- (6,1.02);
\end{scope}
\end{tikzpicture}
\caption{Illustration of the set $\Lapprox$ (white region). Left: Full schematic view. Right: Focus on $ \{\lmin\} + [0,1]^K$. The dashed lines indicate the boundary of the set $\bLapprox = \phi(\bWapproxone)$ considered in the proof of Lemma~\ref{lem:BoundsonLapprox}.}
\label{fig:Lapprox}
\end{figure}
\noindent
With $\varepsilon' = \frac{\varepsilon}{2}$ and $\beta = (1 + \varepsilon')$, Corollary~\ref{cor:LapproxApproximationGuarantee} states that the set~$\Lapprox$ indeed satisfies Property~\ref{property:LApprox}. 

\bigskip

\noindent
Now, to prove Property~\ref{property:Cardinality}, the following lemma provides useful upper and lower bounds on~$\Lapprox$.
\begin{lemma}\label{lem:BoundsonLapprox}
	For $0< \varepsilon' <1$, $\beta \geq 1$, and~$c$ defined as in~\eqref{eq:c}, it holds that
	\begin{align*}
	\Lapprox \subseteq \{\lmin\} + \left[\frac{c^K}{(K+1)!}, \frac{(K+1)!}{c^K} \right]^K.
	\end{align*}
	In particular, $\Lapprox$ is compact.
\end{lemma}
\begin{proof}
Let $\bWapproxone \subseteq W_1$ be defined as in Corollary~\ref{cor:approxnormedweightset}. Then, since $\Wapproxone \subseteq \bWapproxone$, we have
\begin{align*}
\Lapprox = \phi \left( \Wapproxone \right) \subseteq \phi \left( \bWapproxone\right). 
\end{align*}
Also, note that $\bWapproxone = \conv \left(\{ \bar{w}^0, \dots, \bar{w}^K \} \right)$, where, for $i,k \in \{0,\dots,K\}$,
\begin{align*}
\bar{w}^k_i = 
\begin{cases} 
1 - \frac{K}{(K+1)!} \cdot c^K, &\text{for } i = k,\\
\frac{1}{(K+1)!} \cdot c^K, &\text{otherwise}.
\end{cases}
\end{align*}
Thus, for any parameter~$\lambda \in \Lapprox$, there exist scalars $\theta_{0}, \theta_1 \dots, \theta_K \in [0,1]$ with $\sum_{k = 0}^K \theta_k = 1$ such that $\lambda = \phi \left( \sum_{k=0}^K \theta_k \bar{w}^k\right)$. Consequently, for $i=1,\dots, K$, both
\begin{align*}
\lambda_i - \lmin_i = \frac{\sum_{k = 0}^K \theta_k \bar{w}_i^k}{\sum_{k = 0}^K \theta_k \bar{w}_0^k} \geq \frac{ \frac{1}{(K+1)!} \cdot c^K \cdot \sum_{k = 0}^K \theta_k}{\left(1 - \frac{K}{(K+1)!}  \cdot c^K \right) \cdot \sum_{k = 0}^K \theta_k} = \frac{c^K}{(K+1)! - K \cdot c^K} \geq \frac{c^K}{(K+1)! }
\end{align*}
and
\begin{align*}
\lambda_i - \lmin_i = \frac{\sum_{k = 0}^K \theta_k \bar{w}_i^k}{\sum_{k = 0}^K \theta_k \bar{w}_0^k} 
\leq \frac{\left(1 - \frac{K}{(K+1)!} \cdot c^K \right) \cdot \sum_{k = 0}^K \theta_k}{ \frac{1}{(K+1)!} \cdot c^K \cdot \sum_{k = 0}^K \theta_k} 
= \frac{(K+1)!}{c^K} - K \leq \frac{(K+1)! }{c^K}
\end{align*}
hold, which shows the claim. 
\end{proof} 
Next, we construct a grid~$\Grid \subseteq \Lambda$ possessing Properties~\ref{property:Cardinality} and~\ref{property:Grid}. That is, the cardinality is polynomially bounded in the encoding length of the instance and $\frac{1}{\varepsilon}$, and computing a $((1 + \frac{\varepsilon}{2}) \cdot \alpha)$-approximation for any $\lambda\in \Lapprox$ is possible by computing an $\alpha$-approximation for each grid point~$\lambda' \in \Grid$.

Let $c = \frac{\varepsilon \cdot \LB}{2 \cdot (1+\varepsilon') \cdot \alpha \dot \UB}$ be defined as in~\eqref{eq:c} with $\varepsilon' = \frac{\varepsilon}{2}$ and $\beta = (1 + \varepsilon') \cdot \alpha$. We employ the bounds on~$\Lapprox$ given by Lemma~\ref{lem:BoundsonLapprox}, and define a lower bound as well as an upper bound by
\begin{align}\label{eq:lbub}
\lb &\colonequals \left\lfloor \log_{1+\frac{\varepsilon}{2}} \frac{c^K}{(K+1)!} \right\rfloor \text { and }
\ub \colonequals \left\lceil \log_{1+\frac{\varepsilon}{2}} \frac{(K+1)!}{c^K} \right\rceil.
\end{align}
We then set
\begin{align}\label{eq:G}
\Grid \colonequals \left\{\Lambda \in \Lambda: \lambda = \left( \lmin_1 + \left(1 + \frac{\varepsilon}{2}\right)^{i_1}, \dots, \lmin_K + \left(1 + \frac{\varepsilon}{2}\right)^{i_K}  \right)^\top, i_k \in \Z, \lb \leq i_k \leq \ub , k = 1, \dots, K \right\}.
\end{align}
Now, Property~\ref{property:Cardinality} can be shown using the construction of~$\Grid$:
\begin{proposition}\label{lem:numberGridPoints}
	Let~$\Grid$ be defined as in~\eqref{eq:G}. Then, 
	\begin{align*}
	|\Grid| \in \mathcal{O} \left( \left( \frac{1}{\varepsilon} \cdot \log \frac{1}{\varepsilon} + \frac{1}{\varepsilon} \cdot \log \frac{\UB}{\LB} + \frac{1}{\varepsilon} \cdot \log \alpha \right)^K \right).
	\end{align*}
\end{proposition}
\begin{proof}
We have $|\Grid| = (\ub - \lb +1)^K$, where 
\begin{align*}
\ub - \lb +1 &= \left\lceil \log_{1+\frac{\varepsilon}{2}} \frac{(K+1)!}{c^K} \right\rceil -   \left\lfloor \log_{1+\frac{\varepsilon}{2}} \frac{c^K}{(K+1)!} \right\rfloor + 1 \\
&< 2 \cdot \log_{1 + \frac{\varepsilon}{2}} \frac{(K+1)!}{c^K} + 3 \\
&\in \mathcal{O} \left( \log_{1 + \frac{\varepsilon}{2}} \frac{1}{c}\right) \\
&= \mathcal{O} \left( \log_{1 + \frac{\varepsilon}{2}} \frac{2 \cdot (1 + \frac{\varepsilon}{2}) \cdot \alpha \cdot \UB}{\varepsilon \cdot \LB} \right)\\
&=\mathcal{O} \left( \log_{1 + \frac{\varepsilon}{2}} \frac{4 \cdot \alpha \cdot \UB}{\varepsilon \cdot \LB} \right)\\
&=\mathcal{O} \left( \frac{1}{\varepsilon} \cdot \log \frac{1}{\varepsilon} + \frac{1}{\varepsilon} \cdot \log \frac{\UB}{\LB} + \frac{1}{\varepsilon} \cdot \log \alpha \right),
\end{align*}
since~$0 < \varepsilon < 1$.
Here, note that $a^{\frac{\varepsilon}{2}} = a^{(1 - \frac{\varepsilon}{2}) \cdot 0 + \frac{\varepsilon}{2} \cdot 1} \leq (1 - \frac{\varepsilon}{2}) \cdot a^0 + \frac{\varepsilon}{2} \cdot a^1 = 1 + \frac{\varepsilon}{2}$ by convexity of exponential functions with base $a >0$, which implies $\frac{\varepsilon}{2} \leq \log(1 + \frac{\varepsilon}{2})$.  
\end{proof}

\bigskip
It remains to prove that $\Grid$ indeed satisfies Property~\ref{property:Grid}, for which the main idea is motivated by the approximation of multi-objective optimization problems, cf.~\cite{Papadimitriou+Yannakakis:multicrit-approx}.
\begin{proposition}\label{prop:apprximation regionLambda}
	Let 
    $\bar{\lambda} \in \Lambda$ such that $\bar{\lambda}_i > \lmin_i$ for $i = 1, \dots,K$. If~$x \in X$ is an $\alpha$-approximation for~$\bar{\lambda}$, then~$x$ is a $((1+\frac{\varepsilon}{2}) \cdot \alpha)$-approximation for all parameter vectors
	\begin{align*}
	\lambda \in &\left\{ 
	\lambda' : \bar{\lambda}_k - \lmin_k \leq \lambda'_k - \lmin_k \leq \left(1+\frac{\varepsilon}{2}\right) \cdot (\bar{\lambda}_k - \lmin_k), k=1,\dots,K \right\}.
	\end{align*}
\end{proposition}
\begin{proof}
	First let $\lambda \in \R^K$ such that $\bar{\lambda}_k - \lmin_k \leq \lambda_k -\lmin_k \leq (1+\frac{\varepsilon}{2}) \cdot (\bar{\lambda}_k - \lmin_k)$ for $k=1,\dots, K$. Then, for any $x' \in X$, it holds that
	\begin{align*}
		f(x',\lambda) 
		&= f(x',\lmin) + \sum_{k = 1}^K (\lambda_k - \lmin_k) \cdot b_k(x')\\
		&\leq f(x',\lmin) + \sum_{k = 1}^K (1+\frac{\varepsilon}{2}) \cdot (\bar{\lambda}_k - \lmin_k) \cdot b_k(x')\\
		&\leq (1+\frac{\varepsilon}{2}) \cdot \left( f(x',\lmin) + \sum_{k = 1}^K (\bar{\lambda}_k - \lmin_k) \cdot b_k(x') \right) \\
		&= (1+\frac{\varepsilon}{2}) \cdot f(x',\bar{\lambda})
		\leq (1+\frac{\varepsilon}{2}) \cdot \alpha \cdot f(x,\bar{\lambda})  \\
		&= (1+\frac{\varepsilon}{2}) \cdot \alpha \cdot \left( f(x,\lmin) + \sum_{k = 1}^K (\bar{\lambda}_k - \lmin_k) \cdot b_k(x) \right) \\
		&\leq (1+\frac{\varepsilon}{2}) \cdot \alpha \cdot \left( f(x,\lmin) + \sum_{k = 1}^K (\lambda_k - \lmin_k) \cdot b_k(x) \right) = (1+\frac{\varepsilon}{2}) \cdot \alpha \cdot f(x, \lambda).
	\end{align*}
\end{proof}
Note that $\Grid$ is constructed in a way such that, for any parameter vector~$\lambda' \in \Lapprox$, there exists a parameter vector~$\bar{\lambda} \in \Grid$ satisfying $\bar{\lambda}_k - \lmin_k \leq \lambda'_k - \lmin_k \leq (1 + \frac{1}{\varepsilon}) (\bar{\lambda}_k - \lmin_k)$ for $k=1,\dots,K$. Hence, Property~\ref{property:Grid} follows immediately by Proposition~\ref{prop:apprximation regionLambda}. This concludes the discussion of the details. 

\bigskip
\noindent
Our general approximation method for multi-parametric optimization problems is now obtained as follows: Given an instance~$\Pi$, an $\alpha$-approximation algorithm~$\ALG_\alpha$ for the non-parametric version, and~$\varepsilon >0$, we construct the grid~$\Grid$ defined in~\eqref{eq:G}, apply~$\ALG_\alpha$ for each parameter vector~$\lambda \in \Grid$, and collect all solutions in a set~$S$. Since, as shown before, Properties~\ref{property:LApprox}--\ref{property:Grid} hold true, the set~$S$ is indeed a $((1+\varepsilon)\cdot \alpha)$-approximation set. Algorithm~\ref{alg:gridapproach} summarizes the method.

\begin{algorithm}[H]
	\SetKw{Compute}{compute}
	\SetKw{Break}{break}
	\SetKwInOut{Input}{input}\SetKwInOut{Output}{output}
	\SetKwComment{command}{right mark}{left mark}
	
	\Input{An instance~$\Pi$ of a multi-parametric optimization problem, $\varepsilon >0$, an $\alpha$-approximation algorithm~$\ALG_\alpha$ for the non-parametric version of~$\Pi$.}
	
	\Output{A $((1+\varepsilon) \cdot \alpha)$-approximation set for~$\Pi$.}
	
	\BlankLine
	Compute $\LB$ and $\UB$.
	
	Compute $\LB$ and $\UB$.
	
	$\varepsilon' \leftarrow \frac{\varepsilon}{2}$.
	
	$\beta \leftarrow (1 + \varepsilon') \cdot \alpha$.
	
	Set $c$ as in \eqref{eq:c}
	
	Set $\lb$, $\ub$ as in \eqref{eq:lbub}
	
	Set $\Grid$ as in \eqref{eq:G}.
	
	\For{$\lambda \in \Grid$}{
		
		$x \leftarrow  \ALG_{\alpha}(\lambda)$\\
		
		$S \leftarrow S \cup \{x\}$ \\
	}
	\Return $S$.
	\caption{Grid approach for the approximation of multi-parametric optimization problems.}
	\label{alg:gridapproach}
\end{algorithm}

\begin{theorem}\label{th:runningtime}
	Algorithm~\ref{alg:gridapproach} returns a $((1 + \varepsilon) \cdot \alpha)$-approximation set~$S$
	in time$$ \mathcal{O} \left( T_{\LB / \UB} + T_{\ALG_{\alpha}} \cdot  \left( \frac{1}{\varepsilon} \cdot \log \frac{1}{\varepsilon} + \frac{1}{\varepsilon} \cdot \log \frac{\UB}{\LB} + \frac{1}{\varepsilon} \cdot \log \alpha \right)^K \right),$$
	where $T_{\LB / \UB}$ denotes the time needed for computing the bounds~$\LB$ and~$\UB$, and $T_{\ALG_{\alpha}}$ denotes the running time of~$\ALG_{\alpha}$.
\end{theorem}
\begin{proof}
By Lemma \ref{lem:numberGridPoints}, the number of iterations and, thus, the number of calls to $\ALG_{\alpha}$ is asymptotically bounded by 
$$\mathcal{O} \left( \left( \frac{1}{\varepsilon} \cdot \log \frac{1}{\varepsilon} + \frac{1}{\varepsilon} \cdot \log \frac{\UB}{\LB} + \frac{1}{\varepsilon} \cdot \log \alpha \right)^K \right).$$

Now, it remains to show that the set~$S$ returned by the algorithm is a $((1+ \varepsilon) \cdot \alpha)$-approximation set, i.e., that, for each parameter vector~$\lambda \in \Lambda$, there exists a parameter vector~$\bar{\lambda} \in \Grid$ such that any $\alpha$-approximation for $\bar{\lambda}$ is a $((1+\varepsilon) \cdot \alpha)$-approximation for~$\lambda$.
Let $\lambda \in \Lambda$ be a parameter vector. By Corollary~\ref{cor:LapproxApproximationGuarantee}, there exists a parameter vector $\lambda' \in \Lapprox$ such that any $((1+\frac{\varepsilon}{2}) \cdot \alpha)$-approximation for $\lambda'$ is a $((1 + \frac{\varepsilon}{2})\cdot \alpha + \frac{\varepsilon}{2})$-approximation, and thus a $((1+\varepsilon)\cdot \alpha)$-approximation for $\lambda$. We set
$$\bar{\lambda}_i \colonequals \lmin_i + (1 + \varepsilon)^{m_i} \text{ with } m_i \colonequals \lfloor \log_{1 + \varepsilon'} (\lambda'_i - \lmin_i) \rfloor.$$
Then, by Lemma~\ref{lem:BoundsonLapprox}, we have $\lb \leq m_i \leq \ub$ for $i = 1, \dots, K$ and, thus, $\bar{\lambda} \in \Grid$. Moreover, $\bar{\lambda}_i -\lmin_i \leq \lambda'_i - \lmin_i \leq (1+\frac{\varepsilon}{2}) \cdot (\bar{\lambda}_i - \lmin_i)$ and, hence, any $\alpha$-approximation for $\bar{\lambda}$ is a $((1+\frac{\varepsilon}{2})\cdot \alpha)$-approximation for $\lambda'$ by Proposition~\ref{prop:apprximation regionLambda}. This concludes the proof.
\end{proof}

\noindent
In particular, Theorem~\ref{th:runningtime} yields:
\begin{corollary}\label{cor:fptas}
	Algorithm~\ref{alg:gridapproach} yields a FPTAS if either an exact algorithm~$\ALG_{1}$ or an FPTAS is available for the non-parametric version of $\Pi$. If a PTAS is available for the non-parametric version, Algorithm~\ref{alg:gridapproach} yields a PTAS.\end{corollary}
\begin{proof}
If an exact algorithm~$\ALG_{1}$ is available, the statement directly follows from Theorem~\ref{th:runningtime}. Otherwise, for any $\varepsilon>0$, set~$\delta \colonequals \sqrt{1+ \varepsilon} -1$. Then, by Theorem~\ref{th:runningtime}, we can compute a $(1+\varepsilon)$-approximation set in time 
$$ \mathcal{O} \left( T_{\LB / \UB} + T_{\ALG_{1 + \delta}} \cdot  \left( \frac{1}{\delta} \cdot \log \frac{1}{\delta} + \frac{1}{\delta} \cdot \log \frac{\UB}{\LB} \right)^K \right).$$ 
\end{proof}

\section{Minimum-Cardinality Approximation Sets}\label{sec:minCard}
In this section, the task of finding a $\beta$-approximation set~$S^*$ with minimum cardinality is investigated. It is stated in~\cite{Vassilvitskii+Yannakakis:trade-off-curves} that no constant approximation factor on the cardinality of~$S^*$ can be achieved in general for multi-parametric optimization problems with positive parameter set and positive, polynomial-time computable functions~$a, b_k$ if only $(1+\delta)$-approximation algorithms for $\delta>0$ are available for the non-parametric problem. Thus, the negative result also holds in the more general case considered here.
\begin{theorem}
	\label{th:notccompetitive}
	For any~$\beta > 1$ and any integer~$L \in \N$, there does not exist an algorithm that computes a $\beta$-approximation set~$S$ such that $|S| < L \cdot |S^*|$ for every $3$-parametric minimization problem and generates feasible solutions only by calling~$\ALG_{1 + \delta}$ for values of~$\delta>0$ such that~$\frac{1}{\delta}$ is polynomially bounded in the encoding length of the input. 
\end{theorem}
We remark that the corresponding proof (published in~\cite{Diakonikolas2011approximation}, Theorem 5.4.12) is imprecise, but the idea remains valid with a more careful construction. We provide a counterexample and a correction of the proof in the appendix.\\

Note that this result does not rule out the existence of a method that achieves a constant factor if a polynomial-time \emph{exact} algorithm~$\ALG_{1}$ is available. We now show that, in this case, there cannot exist a method that yields an approximation factor smaller than~$K+1$ on the cardinality of~$S^*$. 
\begin{theorem}
	For any $\beta>1$ and $K \in \N$, there does not exist an algorithm that computes a $\beta$-approximation set~$S$ with $|S| < (K+1) \cdot |S^*|$ for every $K$-parametric minimization problem and generates feasible solutions only by calling~$\ALG_{1}$.
\end{theorem}
\begin{proof}
Let $\beta > 1$. In the following, an instance of the augmented multi-parametric optimization problem with parameter set given by the bounded $K$-dimensional simplex~$W_1$ is constructed such that the minimum-cardinality $\beta$-approximation set~$S^*$ has cardinality one, but the unique solution~$x \in S^*$ cannot be obtained by~$\ALG_{1}$, and any other $\beta$-approximation set must have cardinality greater than or equal to~$K+1$. Consider an instance with $X = \{x, x^0,\dots,x^K\}$ such that
\begin{align*}
F_i(x) &=(K+1) \cdot \beta \text{ for } i = 0,\dots,K,
\end{align*}
and, for $i = 0,\dots,K$,
\begin{align*}
F_i(x^i) &= K+1 \text{ and }
F_j(x^i) = (K+2) \cdot \beta - 1 \text{ for }j \neq i.
\end{align*}
We show that the solution~$x$ cannot be obtained via~$\ALG_{1}$, the set~$\{x\}$ a $\beta$-approximation set, and the only $\beta$-approximation set that does not contain~$x$ is $\{x^0, \dots, x^K\}$ with cardinality~$K+1$. 

First, we show that~$x$ cannot be obtained via~$\ALG_{1}$. For any~$w \in W_1$, there exists an index~$i \in \{0,\dots,K\}$ such that $w_i \geq \frac{1}{K+1}$ and, thus,
\begin{align*}
w^\top F(x^i) &= (K+1) \cdot w_i + ( (K+2) \cdot \beta - 1) \cdot (1 - w_i) \\
&= (K+2) \cdot \beta - 1 + (K+1 - (K+2) \cdot \beta + 1) \cdot w_i \\
&= (K+2) \cdot \beta - 1 + (K+2 - (K+2) \cdot \beta ) \cdot w_i \\
&\leq (K+2) \cdot \beta - 1 + \frac{K+2}{K+1} \cdot (1 - \beta) \\
&= (K+2) \cdot \beta - 1 + 1- \beta + \frac{1}{K+1} \cdot (1 - \beta) \\
&= (K+1) \cdot \beta + \frac{1}{K+1} \cdot (1 - \beta) \\
&< (K+1) \cdot \beta\\
&= w^\top F(x).
\end{align*}
Hence, the solution~$x$ cannot be obtained via~$\ALG_{1}$. Next, we show that the set~$\{x\}$ is a $\beta$-approximation set. For any~$w \in W_1$ and any $i = 0, \dots, K$, it holds that
\begin{align*}
w^\top F(x) &= \beta \cdot \left( (K+1) \cdot w_i + (K+1) \cdot (1 - w_i) \right) \\
&= \beta \cdot \left( (K+1) \cdot w_i + (K+2 -1)\cdot (1 - w_i) \right) \\
&\leq \beta \cdot \left( (K+1) \cdot w_i + ( (K+2) \cdot \beta - 1) \cdot (1 - w_i) \right) \\
&= \beta \cdot w^\top F(x^i).
\end{align*}
Hence, the solution $x$ is a $\beta$-approximation for any $\lambda \in \Lambda$. Finally, we show that the only $\beta$-approximation set that does not contain~$x$ is $\{x^0, \dots, x^K\}$. Let $e^i \in W_1$ be the $i$th unit vector. Then, for any $j \in \{0, \dots,K\} \setminus \{i\}$, we have
\begin{align*}
\beta \cdot (e^i)^\top F(x^i) &= (K+1) \cdot \beta \\
&< (K+1) \cdot \beta + \beta - 1 \\
&= (K+2) \cdot \beta - 1 = (e^i)^\top F(x^j).
\end{align*}
Note that, by continuity of $w^\top F(x)$ in $w_i$ for $i = 0,\dots, K$, there exists a small~$t > 0$ such that, for each $i \in \{0, \dots, K\}$, the weight~$w^i$ defined by $w^i_i \colonequals 1 - \frac{K \cdot t}{K+1}$ and $w^i_j = \frac{t}{K+1}$, $j \neq i$, satisfies $\beta \cdot (w^i)^\top F(x^i) < (w^i)^\top F(x^j)$ for all $j \neq i$. Hence, the above arguments also hold for weights~$w \in W_1 \cap \R^{K+1}_>$. Therefore, the above instance shows that no $\beta$-approximation set with cardinality less than $K+1$ times the size of the smallest $\beta$-approximation set can be obtained using~$\ALG_{1}$ for linear multi-parametric optimization problems in general. 
\end{proof}

\section{Applications}\label{sec:Applications}
In this section, the established results are applied to linear multi-parametric versions of important optimization problems. The $1$-parametric versions of the shortest path problem, the assignment problem, linear mixed-integer programs, the minimum cost flow problem, and the metric traveling salesman problem have previously been covered in~\cite{Bazgan+etal:parametric}. By employing Theorem~\ref{th:runningtime}, it is easy to see that the stated results generalize to the multi-parametric case in a straightforward manner. We now apply Theorem~\ref{th:runningtime} to several other well-known problems. Note that, for a maximization problem and some~$\beta \geq 1$, a $\beta$-approximate solution for the non-parametric version~$\Pi(\lambda)$ is a feasible solution~$x \in X$ such that $f(x,\lambda) \geq \frac{1}{\beta} \cdot f(x',\lambda)$ for all $x' \in X$.

\medskip

\textbf{Multi-Parametric Minimum $s$-$t$-Cut Problem}
Given a directed graph $G=(V,R)$ with $|V|=n$ and $|R| = m$, a multi-parametric cost function $a_r + \sum_{k=1}^K b_{k,r}$ for each $r \in R$, where $a_r, b_{k,r} \in \N_{0}$, and two vertices $s,t \in V$ with $s \neq t$, the \emph{multi-parametric minimum $s$-$t$-cut problem} asks to compute an $s$-$t$-cut $(S_{\lambda},T_{\lambda})$, $s \in S_{\lambda}$ and $t \in T_{\lambda}$, of minimum total cost $\sum_{r:\alpha(r)\in S_{\lambda},\omega(r)\in T_{\lambda}} a_r + \sum_{k=1}^K \lambda_k b_r$ for each $\lambda \in \Lambda$ (where $\alpha(r)$ denotes the start vertex and $\omega(r)$ the end vertex of an arc~$r\in R$). Here, $\lmin$ can be defined by setting $\lmin_k \colonequals \max_{r \in R} \{ - \frac{a_r}{K \cdot b_{k,r}}: b_{k,r} \neq 0\}$ such that, for each parameter vector greater than or equal to~$\lmin$, the cost of each $s$-$t$-cut is nonnegative.

A positive rational upper bound~$\UB$ as in Assumption~\ref{ass:poly} can be obtained by summing up the $m$~cost components~$a_r$ and summing up the $m$~cost components~$b_{k,r}$ for each~$k$, and taking the maximum of these $K+1$~sums. The lower bound~$\LB$ can be chosen as~$\LB\colonequals 1$. The non-parametric problem can be solved in 
$\mathcal{O} \left( n \cdot m \right)$ for any fixed~$\lambda$ (cf.~\cite{Orlin:STOC13}).
Hence, an FPTAS for the multi-parametric minimum $s$-$t$-cut problem with running time $\mathcal{O}\left( (n \cdot m)(\frac{1}{\varepsilon} \log \frac{1}{\varepsilon} + \frac{1}{\varepsilon} \log(mC) )^{K} \right)$ is obtained, where~$C$ denotes the maximum value among all~$a_r,b_{k,r}$.

The number of required solutions in an optimal solution set can be super-polynomial even for $K=1$~\cite{Carstensen:parametric}. Remarkably, a recent result shows that the number of required solutions in an optimal solution set of the $K$-parametric minimum $s$-$t$-cut problem with $K >1$ can be exponential even for instances that satisfy the so-called \emph{source-sink-monotonicity}~\cite{Allman:complexity2parameterMinCut}, whereas instances of the $1$-parametric minimum $s$-$t$-cut problem satisfying source-sink-monotonicity can be solved exactly in polynomial time~\cite{Gallo+etal:parametric-flow,McCormick:parametric-max-flow}. Consequently, in the multi-parametric case, an FPTAS is the best-possible approximation result even for instances satisfying source-sink-monotonicity. 

\medskip

\textbf{Multi-Parametric Maximization of Independence Systems} Let a finite set $E=\{1,\dots,n\}$ of elements and a nonempty family~$\mathcal{F} \subseteq 2^E$ be given. The pair $(E,\mathcal{F})$ is called an \emph{independence system} if $\emptyset \in \mathcal{F}$ and, for each set~$x \in \mathcal{F}$, it follows that all its subsets $x' \subseteq x$ are also contained in $\mathcal{F}$. The elements of~$F$ are then called \emph{independent sets}. 
The \emph{lower rank}~$l(F)$ and the \emph{upper rank}~$r(F)$ of a subset~$F \subseteq E$ of elements are defined by $l(F) \colonequals \min\{ |B|: B \subseteq E, B \in \mathcal{F} \text{ and } B \cup \{e\} \notin \mathcal{F} \text{ for all } e \in E \setminus F \}$ and $r(F)  \colonequals \max\{|B| : B \subseteq E, B \in \mathcal{F} \}$, respectively. The \emph{rank quotient}~$q(E,\mathcal{F})$ of the independence system~$(E, \mathcal{F})$ is then defined as $q(E,\mathcal{F}) \colonequals \min_{F \subseteq E, l(F) \neq 0} \frac{r(F)}{l(F)}$.
Moreover, let a multi-parametric cost of the form $a_e + \sum_{k=1}^K \lambda_k b_{k,e}$, where $a_e,b_{k,e} \in \N_0$, $k=1,\dots,K$, be given for each element~$e \in E$. Then, with~$\lmin$ defined by $\lmin_k \colonequals \max_{e \in E} \{ - \frac{a_e}{K \cdot b_{k,e}}: b_{k,e} \neq 0 \}$, $k=1,\dots,K$, the \emph{multi-parametric maximization of independence systems problem} asks to compute, for each parameter vector~$\lambda$ greater than or equal to~$\lmin$, an independent set $x_\lambda \in \mathcal{F}$ of maximum cost $\sum_{e \in E} a_e + \sum_{k=1}^K \lambda_k b_{k,e}$.

Here, a positive rational upper bound~$\UB$ as in Assumption~\ref{ass:poly} can be obtained by summing up the $n$~profit components~$a_e$ and summing up the $n$~profit components~$b_{k,e}$ for each~$k$, and taking the maximum of these $K+1$~sums. The lower bound~$\LB$ can again be chosen as $\LB\colonequals 1$. 
For independence systems~$(E,\mathcal{F})$ with rank quotient~$q(E,\mathcal{F})$, it is known that the greedy algorithm is a $q(E,\mathcal{F})$-approximation algorithm for the non-parametric problem obtained by fixing any parameter vector~$\lambda$~\cite{Korte+Haussmann:GreedyIndSystems}. Hence, for any $\epsilon>0$, the maximization version of Theorem~\ref{th:runningtime} yields a $((1+\epsilon)\cdot q(E,\mathcal{F}))$-approximation algorithm with running time $\mathcal{O} \left( T_{\text{Greedy}} \cdot (\frac{1}{\varepsilon} \log \frac{1}{\varepsilon} + \frac{1}{\varepsilon} \log(nC) )^{K} \right)$, where~$C$ denotes the maximum profit component among all~$a_e, b_{k,e}$, and~$T_{\text{Greedy}}$ denotes the running time of the greedy algorithm (which can often be seen to be in $\mathcal{O}(n \log(n)$ times the running time of deciding whether $F \in \mathcal{F}$ holds for any set~$F \subseteq E$ of elements). Since the maximum matching problem (with the assignment problem as a special case) in an undirected graph~$G = (V,E)$ constitutes a special case of the maximization of independence systems problem~\cite{Korte+Haussmann:GreedyIndSystems}, the number of required solutions in an optimal solution set for the multi-parametric maximization of independence systems problem can be super-polynomial in~$n$ even for~$K=1$~\cite{Carstensen:parametric,Carstensen:PHD}.

\medskip

For example, our result yields a $((1 + \varepsilon) \cdot 2)$-approximation algorithm for the multi-parametric $b$-matching problem and a $((1 + \varepsilon) \cdot 3)$-approximation algorithm for the multi-parametric maximum asymmetric TSP (cf.~\cite{Mestre:GreedyinApproximation}). Note that the knapsack problem can also be formulated using independence systems. For this problem, an approximation scheme for the non-parametric version is known:

\medskip

\textbf{Multi-Parametric Knapsack Problem} Let a set~$E = \{1,\dots,n\}$ of items and a budget~$W \in \N_0$ be given. Each item~$e \in E$ has a multi-parametric profit of the form $a_e + \sum_{k=1}^K \lambda_k \cdot b_{k,e}$, where $a_e, b_{k,e} \in \N_0$, $k=1,\dots,K$, are nonnegative integers, and a weight~$w_e \in \N_0$. The parameter vector~$\lmin$ is chosen by $\lmin_k \colonequals \max_{e \in E} \{ - \frac{a_e}{K \dot b_{k,e}}: b_{k,e} \neq 0 \}$, $K=1, \dots,K$, such that, for each set~$x \subseteq \{1,\dots,n\}$ of items, the profit components $\sum_{e \in x} a_e$ and $\sum_{e\in x} b_{k,e}$, $k = 1,\dots,K$, are nonnegative. Then, the multi-parametric knapsack problem asks to compute a subset~$x \subseteq E$ satisfying~$\sum_{e \in x} w_e \leq W$ of maximum profit for each parameter vector~$\lambda$ greater than or equal to~$\lmin$.

For this problem, a positive rational upper bound~$\UB$ as in Assumption~\ref{ass:poly} can again be obtained by summing up the $n$~profit components~$a_e$ and summing up the $n$~profit components~$b_{k,e}$ for each~$k$, and taking the maximum of these $K+1$~sums. The lower bound~$\LB$ can again be chosen as $\LB\colonequals 1$. The currently best approximation scheme for the non-parametric problem is given in~\cite{Kellerer1999new,Kellerer:ImprovedDPwithFPTAS}, which computes, for any~$\varepsilon' > 0$, a feasible solution whose profit is no worse than $(1-\varepsilon')$ times the profit of any other feasible solution in time $ \mathcal{O}\left( n \cdot \min \left\{\log n, \log \frac{1}{\varepsilon'} \right\} + \frac{1}{\varepsilon'}\log \frac{1}{\varepsilon'}\cdot \min \left\{ n , \frac{1}{\varepsilon'} \log \frac{1}{\varepsilon'} \right\} \right)$.
Assuming that~$n$ is much larger than $\frac{1}{\varepsilon'}$ (cf.~\cite{Kellerer+etal:book}), choosing $\varepsilon'= \frac{1}{2} \cdot \varepsilon$ and applying the maximization version of Theorem~\ref{th:runningtime} with $\frac{1}{2} \cdot \varepsilon$ yields an FPTAS for the multi-parametric knapsack problem with running time
$\mathcal{O}\left(  (n \log \frac{1}{\varepsilon} + \frac{1}{\varepsilon^3} \log^2 \frac{1}{\varepsilon} ) (\frac{1}{\varepsilon} \log \frac{1}{\varepsilon} + \frac{1}{\varepsilon} \log(nC) )^{K} \right) $, where~$C$ denotes the maximum profit component among all~$a_e, b_{k,e}$.

Again, the number of required solutions in an optimal solution set can be super-polynomial in~$n$ even for $K=1$~\cite{Carstensen:PHD}.

\section{Conclusion}
Exact solution methods, complexity results, and approximation methods for multi-parametric optimization problems are of major interest in recent research. In this paper, we establish that approximation algorithms for many important non-parametric optimization problems can be lifted to approximation algorithms for the multi-parametric version of such problems. The provided approximation guarantee is arbitrarily close to the approximation guarantee of the non-parametric appro\-ximation algorithm. This implies the existence of a multi-parametric FPTAS for many important multi-parametric optimization problems for which optimal solution sets require super-polynomially many solutions in general.

Moreover, our results show that computing an approximation set containing the smallest-possible number of solutions is not possible in general. However, practical routines to reduce the number of solutions in the approximation set, based for example on the convexity property of Lemma~\ref{prop:alphaconvex} and the approximation method in~\cite{Bazgan+etal:parametric}, might be of interest.
Another direction of future research could be the approximation of multi-parametric MILPs with parameter dependencies in the constraints. Here, relaxation methods or a multi-objective multi-parametric formulation of the problems may provide a suitable approach. 
%

\bibliographystyle{spmpsci}
\bibliography{Literatur}

\section*{Appendix}
	\label{appA}
	The proof of Theorem~\ref{th:notccompetitive} stated in~\cite{Diakonikolas2011approximation}, Theorem~5.4.12, is imprecise, as the following example shows.
	\begin{example}
	We define two instances~$A^1$ and~$A^2$ of the augmented multi-parametric problem of a  multi-parametric optimization problem with parameter set~$\Lambda = \R^2_\geqq$ following the construction presented in the proof of Theorem~5.4.12 in~\cite{Diakonikolas2011approximation}. Instance~$A^1$ has feasible set $X^1 = \{x,x^1,x^2,x^3\}$ and instance~$A^2$ has feasible set $x^2 =\{x,x^1,x^2,x^3,\bar{x}^1,\bar{x}^2, \bar{x}^3\}$. Both instances have the same objective function
	\begin{align*}
	w^\top F(x) = w_0 \cdot F_0(x) + w_1 \cdot F_1(x) + w_2 \cdot F_2(x).
	\end{align*}
	For given $\beta > 1$ and $z_0 >1$, we construct a set~$Q = \{x^1,x^2,x^3\}$ such that $F_0(x^1) = F_0(x^2) = F_0(x^3) = z_0$, and for which the smallest $\beta$-approximation set for an instance with feasible set~$Q$ is~$Q$ itself. Then, we construct a solution~$x$ such that $F_0(x) = \beta \cdot z_0$, $\beta \cdot F_1(x) < F_1(x^\ell)$, and $\beta \cdot F_2(x) < F_2(x^\ell)$ for $\ell = 1,2,3$. Thus, $\{x\}$ is a smallest $\beta$-approximation set for~$A^1$. Finally, we choose~$\bar{x}^\ell$ such that $F_0(\bar{x}^l) = z_0 -1$, $F_{1} (\bar{x}^l) = F_{1}(x^l)$, and $F_{2} (\bar{x}^l) = F_{2}(x^l)$ for $\ell =1,2,3$. We have to show that, for large~$z_0$, the set~$\{x,\bar{x}^1, \bar{x}^3\}$ is a $\beta$-approximation set for~$A^2$. Then, $\{x,\bar{x}^1, \bar{x}^2, \bar{x}^3\}$ is not a smallest $\beta$-approximation set as claimed in~\cite{Diakonikolas2011approximation}.\\
	
	\noindent
	Let~$\beta > 1$ and~$z_0 > 1$ and define 
	\begin{align*}
	F_0(x) \colonequals \beta \cdot z_0,&&F_1(x) &\colonequals 1,&& F_2(x)\colonequals 1,\\
	F_0(x^1) \colonequals z_0,&&F_1(x^1) &\colonequals \beta^2,&& F_2(x^1) \colonequals 2\beta^6 - \beta^2, \\
	F_0(x^2) \colonequals z_0,&&F_1(x^2) &\colonequals \beta^4,&& F_2(x^2) \colonequals \beta^4,\\
	F_0(x^3) \colonequals z_0,&&F_1(x^3) &\colonequals 2\beta^6 - \beta^2,&&F_2(x^3) \colonequals \beta^2.
	\end{align*}
	Then, $\beta \cdot F_i(x) < F_i(x^\ell)$ for $i=1,2$ and $\ell = 1,2,3$.
	Further, define
	\begin{align*}
	F_0(\bar{x}^\ell) \colonequals z_0 -1,&&F_1(\bar{x}^\ell) &\colonequals F_1(x^\ell),&&F_2(\bar{x}^\ell) \colonequals F_2 (x^\ell), 
	\end{align*}
	for $\ell=1,2,3$. In the following, we show that:
	\begin{enumerate}
		\item The set $\{x^1,x^2,x^3\}$ is a smallest $\beta$-approximation set for the instance with solution set $\{x^1,x^2,x^3\}$ and objective $w^\top F(x)$. That is, there exists weights $w^1,w^2,w^3 \in \R^3_\geq$ such that
		\begin{align*}
		\beta (w^1)^\top F(x^1) < (w^1)^\top F(x^2)\text{ and } \beta \cdot (w^1)^\top F(x^1) < (w^1)^\top F(x^3),\\
		\beta (w^2)^\top F(x^2) < (w^2)^\top F(x^1)\text{ and } \beta \cdot(w^2)^\top F(x^2) < (w^2)^\top F(x^3),\\
		\beta (w^3)^\top F(x^3) < (w^3)^\top F(x^1)\text{ and } \beta \cdot (w^3)^\top F(x^3) < (w^3)^\top F(x^2).
		\end{align*}\label{item:counterexample1}
		\item For $z_0 \geq \frac{\beta^2}{\beta -1} + 1$, the set $\{x,\bar{x}^1,\bar{x}^3\}$ is a $\beta$-approximation set for the instance with solution set $\{x,x^1,x^2,x^3,\bar{x}^1,\bar{x}^2,\bar{x}^3\}$ and objective $w^\top F(x)$. More precisely, for $w \in \R^3_\geq$,
		\begin{itemize}
			\item if $w_0 \leq \frac{\beta^5 -1}{\beta} \cdot (w_1 + w_2)$, then $w^\top F(x) \leq \beta \cdot w^\top F(\bar{x}^2)$, 
			\item if $w_0 \geq \frac{\beta^5 -1}{\beta} \cdot (w_1 + w_2)$ and $w_1 \geq w_2$, then $w^\top F(\bar{x}^1) \leq \beta \cdot w^\top F(\bar{x}^2)$,
			\item if $w_0 \geq \frac{\beta^5 -1}{\beta} \cdot (w_1 + w_2)$ and $w_1 \leq w_2$, then $w^\top F(\bar{x}^3) \leq \beta \cdot w^\top F(\bar{x}^2)$.
		\end{itemize}\label{item:counterexample2}
	\end{enumerate}
	Note that these three statements suffice since $F_i(\bar{x}^\ell) \leq F_i(x^\ell)$ for $i = 1,2,3$ and $\ell = 1,2,3$.\\
	In order to show Statement~\ref{item:counterexample1}, choose  $w^1 \colonequals (2\beta^7 - \beta^3 - \beta^4 + 1, \beta^4 - \beta^3,0)^\top$. Then,
	\begin{align*}
	\beta \cdot (w^1)^\top F(x^1) 
	&= 2\beta^{11} - \beta^7 + \beta^3  - \beta^7\\
	&< 2\beta^{11} - \beta^7 + \beta^4 -\beta^7 \\
	&= (w^1)^\top F(x^2)
	\end{align*}
	and the inequality $\beta \cdot (w^1)^\top F(x^1) < (w^1)^\top F(x^3)$ is equivalent to
	\begin{align*}
	4\beta^{13} - 2\beta^{11} - 2\beta^{10} &- 4\beta^9 + 2\beta^7 + 4\beta^6 - \beta^3 - \beta^2 > 0,
	\end{align*}
	which is satisfied for all $\beta >1$.\\
	Next, choose $w^3 \colonequals (\beta^4 - \beta^3,2\beta^7 - \beta^3 - \beta^4 + 1, 0)^\top$. Then, the inequalities~$\beta \cdot (w^3)^\top F(x^3) < (w^3)^\top F(x^1)$ and $\beta \cdot (w^3)^\top F(x^3) < (w^3)^\top F(x^2)$ can be shown analogously.\\
	Finally, choose $w^2 \colonequals (1,1,0)^\top$. Then,
	\begin{align*}
	\beta \cdot (w^2)^\top F(x^2) = 2\beta^5 < 2\beta^6 = 2\beta^6 - \beta^2 + \beta^2 &= 
	(w^2)^\top F(x^1).
	\end{align*}
	Since $(w^2)^\top F(x^1) = (w^2)^\top F(x^3)$, this proves Statement~\ref{item:counterexample1}.\\
	
	\noindent
	In order to prove Statement~\ref{item:counterexample2}, let $w \colonequals (w_0,w_1,w_2,) \in \R^3_\geq$ with $w_0 \leq \frac{\beta^5 -1}{\beta} \cdot (w_1 + w_2)$. Then,
	\begin{align*}
	w^\top F(x) &= \beta \cdot w_0 \cdot z_0 + w_1 + w_2 \\
	&=  \beta \cdot w_0 \cdot (z_0 - 1) + w_1 + w_2 + \beta \cdot w_0\\
	&\leq \beta \cdot w_0 \cdot (z_0 - 1) + w_1 + w_2 + \beta \cdot \frac{\beta^5 -1}{\beta} (w_1 + w_2)\\
	&=  \beta \cdot w_0  \cdot (z_0 - 1) + \beta^5 \cdot (w_1 + w_2) \\
	&= \beta \cdot w^\top F(\bar{x}^2).
	\end{align*}
	Next, let $w \colonequals (w_0,w_1,w_2)^\top \in \R^3_\geq$ with $w_0 \geq \frac{\beta^5 -1}{\beta} \cdot (w_1 + w_2)$ and $w_1 \geq w_2$. The latter inequality implies that
	\begin{align*}
	&&(1 - \beta^4) \cdot w_1 &\leq ( 1 - \beta^4) \cdot w_2 &&\\
	\Rightarrow&& (1 - \beta^4) \cdot w_1 &\leq ( \beta^4 + 1 - 2\beta^4 ) \cdot w_2 &&\\
	\Rightarrow&& w_1 + 2 (\beta^4 - 1) \cdot w_2 &\leq \beta^4 \cdot (w_1 + w_2) && \\
	\Rightarrow&& \beta^2 \cdot w_1  + (2\beta^6 - \beta^2) \cdot w_2  &\leq  \beta^6 \cdot (w_1 + w_2).&&
	\end{align*}
	Hence, for $z_0 \geq \frac{\beta^2}{\beta -1} +1$,
	\begin{align*}
	w^\top F(\bar{x}^1) &= w_0 \cdot (z_0 -1) +  \beta^2 \cdot w_1  +  (2\beta^6 - \beta^2) \cdot w_2\\
	&\leq w_0 \cdot (z_0-1) + \beta^6 \cdot (w_1 + w_2) \\
	&=  \beta \cdot w_0 \cdot ( z_0-1) + \beta^6  \cdot (w_1 + w_2)  - (\beta - 1) \cdot w_0 \cdot (z_0 -1)\\
	&\leq \beta \cdot w_0 \cdot ( z_0-1) + \beta^6 \cdot (w_1 + w_2) - (\beta - 1) \cdot \frac{\beta^5 - 1}{\beta} \cdot (w_1 + w_2) \cdot \frac{\beta^2}{\beta - 1} \\
	&= \beta \cdot w_0 \cdot ( z_0-1) + w_1 + w_2 \\
	&< \beta \cdot w^\top F(\bar{x}^2).
	\end{align*}
	The remaining statement for $w_0 \geq \frac{\beta^5 -1}{\beta} \cdot (w_1 + w_2)$ and $w_2 \geq w_1$ follows by symmetry. \hfill $\vartriangleleft$
\end{example}
	\noindent
	Nevertheless, as we now shown, the idea of~\cite{Diakonikolas2011approximation} remains valid with a more careful construction.
	
	\begin{proof}
	Given~$L \in \N$ and~$\beta > 1$, two instances~$A^1, A^2$ for an augmented multi-parametric optimization problem of a multi-parametric optimization problem~$\Pi$ are constructed such that a smallest $\beta$-approximation set for~$A^2$ is $L+1$~times as large as a smallest $\beta$-approximation set for~$A^1$ and an algorithm has to call~$\ALG_{1 + \delta}$ for some~$\delta$ with $\frac{1}{\delta}$ exponentially large in the input size in order to distinguish between the two instances.
	
	Let $W = \R^{3}_\geqq$ and $z_0 \geq \beta > 1$. Define solutions $x^*$, $x^\ell$, $\ell = 1, \dots, L$ such that $F_0(x^*) = \beta \cdot z_0$ and $F_0(x^\ell) = z_0$ for~$\ell = 1, \dots, L$, and $\beta \cdot F_1(x^*) < F_1(x^\ell) < 1$  and $\beta \cdot F_2(x^*) < F_2(x^\ell) < 1$ for $\ell = 1, \dots, L$, and such that, for each~$x^\ell$, there exists a weight~$w^\ell \in W$ with
	\begin{align}\label{eq:approxwithminimalsolutionset}
	( (\beta-1) \cdot z_0 + \beta) \cdot (w^\ell_1 F_1(x^\ell) + w^\ell_2 F_2(x^\ell)) < w^\ell_1 F_1(x^m) + w^\ell_2 F_2(x^m)
	\end{align}
	for all $\ell,m \in \{1, \dots,L\}$ with $\ell \neq m$.\footnote{For example 
		$F_1(x^*) \colonequals F_2(x^*) \colonequals (2 z_0 \cdot \beta)^{-2-n}$ and 
		$F_1(x^\ell) \colonequals (2 z_0 \cdot \beta)^{\ell - L}$, 
		$F_2(x^\ell) \colonequals (2 z_0 \cdot \beta)^{1 - \ell}$ 
		with weights $w^\ell_1 \colonequals (2 z_0 \cdot \beta)^{L-\ell}$ and $w^\ell_2 \colonequals (2 z_0 \cdot \beta)^{\ell-1}$.} 
	Further, define solutions $\bar{x}^\ell$, $\ell = 1, \dots, L$ with $F_0(\bar{x}^\ell) = z_0 -1$, $F_1(\bar{x}^\ell) = F_1(x^\ell)$ and $F_2(\bar{x}^\ell) = F_2(x^\ell)$ for $\ell = 1, \dots, L$. Set $X \colonequals \{ x^1, \dots, x^L\}$ and $\bar{X} \colonequals \{\bar{x}^1, \dots, \bar{x}^L\}$.
	
	Let instance~$A^1$ have feasible set $X \cup \{x^*\}$, and instance~$A^2$ have feasible set~$X \cup \bar{X} \cup \{x^*\}$.
	In order to distinguish between the two instances for some weight~$w$, an algorithm must be guaranteed to obtain different results when calling $\ALG_{1 + \delta}$ on~$A^1$ and on~$A^2$. Therefore,$\delta$ and~$w$ have to be chosen such that neither~$x^*$ nor any~$x^\ell$ are a $(1 + \delta)$-approximation for~$w$ in~$A^2$. That is, for some $\ell \in \{1, \dots, L\}$, we must have that $w^\top F(x) > (1 + \delta) \cdot w^\top F(\bar{x}^\ell)$ for all $x \in X \cup \bar{X} \cup \{x^*\} \setminus \{\bar{x}^\ell\}$. In particular, we must have
	\begin{align*}
	w^\top F(x^\ell) > (1 + \delta) w^\top F(\bar{x}^\ell).
	\end{align*}
	This implies that
	\begin{align*}
	\Rightarrow&& z_0 &> (1 + \delta) \cdot (z_0 - 1)&& \\
	\Rightarrow&& \delta &< \frac{1}{z_0 - 1}&&\\
	\Rightarrow&& \frac{1}{\delta} &> z_0 -1.&&
	\end{align*}
	Since $z_0$ might be exponentially large in the instance size, $\delta$ might has to be chosen such that $\frac{1}{\delta}$ is exponential in the instance size in order to distinguish between $A^1$ and $A^2$.
	
	In remains to show that $\{x^*\}$ is a $\beta$-approximation set of minimum cardinality for $A^1$, whereas $\bar{X} \cup \{x^*\}$ is a $\beta$-approximation set of minimum cardinality for $A^2$.
	Since 
		\begin{align*}
		w^\top F(x^*) &= w_0F_0(x^*) +w_1 F_1(x^*) + w_2 F_2(x^*)\\
		&= \beta \cdot w_0F_0(x^\ell) +w_1 F_1(x^*) + w_2 F_2(x^*)\\
		&\leq \beta \cdot (w_0F_0(x^\ell) +w_1 F_1(x^\ell) + w_2 F_2(x^\ell)) = \beta \cdot w^\top F(x^\ell)
		\end{align*}
		for all $\ell  = 1, \dots ,L$ and $w \in W$, the set $\{x^*\}$ is a $\beta$-approximation set with minimum cardinality for instance $A^1$.\label{item:proofapproxwithminimal1}
		
		The set~$\bar{X} \cup \{x^*\}$ is also the $\beta$-approximation set with minimal cardinality for instance~$A^2$:\\
		Clearly, 
		$\bar{X} \cup \{x^*\}  $ is a $\beta$-approximation set for $A^2$ since $\{x^*\}$ is a $\beta$-approximation set for $A^1$.
		Let $d >0$ such that $\beta \cdot F_i(x^*) + d< F_i(\bar{x}^\ell)$ for $i = 1,2$ and $\ell = 1, \dots,L$. Choose $w_0 < \frac{d}{\beta^2 z_0}$ and $w_1 = w_2 = \frac{1}{2}$. Then
		\begin{align*}
		\beta \cdot  w^\top F(x^*) &= \beta^2 \cdot w_0  \cdot z_0 + w_1  \beta \cdot F_1(x^*) + \beta  \cdot w_2 F_2(x^*)\\
		&< d + \frac{1}{2} \cdot \beta \cdot F_1(x^*) + \frac{1}{2} \cdot \beta \cdot F_2(x^*)\\
		&< \frac{1}{2} \cdot F_1(\bar{x}^\ell) + \frac{1}{2} \cdot  F_2(\bar{x}^\ell)\\
		&\leq w^\top F(\bar{x}^\ell)
		\end{align*}
		for $\ell = 1, \dots, L$. Therefore, the solution~$x^*$ must be contained in every $\beta$-approximation set.
		Set $\bar{w}^\ell = \left( \frac{z_0}{z_0 -1} \cdot ( \bar{w}^\ell_1 F_1(\bar{x}^\ell) + \bar{w}^\ell_2 F_2(\bar{x}^\ell)),w^\ell_1, w^\ell_2 \right)^\top \in W$ for $\ell = 1, \dots, L$. Then,
		$$w_0 F_0(\bar{x}^l) = z_0 \cdot \left(\bar{w}^l_1 F_1(\bar{x}^l) + \bar{w}^l_2 F_2(\bar{x}^l) \right).$$
		\enlargethispage{\baselineskip}
		On the one hand, it holds that
		\begin{align*}
			\beta \cdot (\bar{w}^\ell)^\top F(\bar{x}^\ell)
			& = \beta \cdot (1 - z_0) \cdot \left(\bar{w}^\ell_1 F_1(\bar{x}^\ell) +\bar{w}^\ell_2 F_2(\bar{x}^\ell \right) \\
			 & < z_0 \cdot \left(\bar{w}^\ell_1 F_1(\bar{x}^\ell) +\bar{w}^\ell_2 F_2(\bar{x}^\ell)\right) + \bar{w}^\ell_1 F_1(\bar{x}^m) +\bar{w}^\ell_2 F_2(\bar{x}^m)\\
			 & = \bar{w}^\ell_0 F_0(\bar{x}^\ell) + \bar{w}^\ell_1 F_1(\bar{x}^m) +\bar{w}^\ell_2 F_2(\bar{x}^m) \\
			& = \bar{w}^\ell_0 F_0(\bar{x}^m) + \bar{w}^\ell_1 F_1(\bar{x}^m) +\bar{w}^\ell_2 F_2(\bar{x}^m) \\
			& = (\bar{w}^\ell)^\top F(\bar{x}^m)\\
			& \leq (\bar{w}^\ell)^\top F(x^m)
		\end{align*}
		for all $\ell,m =1, \dots, L$ with $\ell \neq m$, where the last inequality holds by~\eqref{eq:approxwithminimalsolutionset}.
		On the other hand, it holds that
		\begin{align*}
		\beta \cdot (\bar{w}^\ell)^\top F(\bar{x}^\ell) & = \beta \cdot \left(\bar{w}^\ell_0 F_0(\bar{x}^\ell) + \bar{w}^\ell_1 F_1(\bar{x}^\ell) +\bar{w}^\ell_2 F_2(\bar{x}^\ell)\right)\\
		& = \beta \cdot (z_0 +1) \cdot \left(\bar{w}^\ell_1 F_1(\bar{x}^\ell) +\bar{w}^\ell_2 F_2(\bar{x}^\ell)\right) \\
		& = \beta \cdot \frac{z_0^2}{ z_0 - 1} \cdot  \left(\bar{w}^\ell_1 F_1(\bar{x}^\ell) +\bar{w}^\ell_2 F_2(\bar{x}^\ell)\right)\\
		& = F_0(x^*) \cdot \frac{z_0}{z_0 - 1} \cdot   \left(\bar{w}^\ell_1 F_1(\bar{x}^\ell) +\bar{w}^\ell_2 F_2(\bar{x}^\ell)\right) \\
		& = \bar{w}^\ell_0 F_0(x^*)\\
		& \leq (\bar{w}^\ell)^\top F(x^*)
		\end{align*}
		for $\ell = 1, \dots, L$. Hence, any $\beta$-approximation set for~$A^2$ must contain either~$x^\ell$ or~$\bar{x}^\ell$ for each~$\ell = 1, \dots, L$, which proves the claim.  
	\end{proof}

\end{document}